\documentclass[11pt, reqno, a4paper]{amsart}

\usepackage{amsmath, amssymb, amsthm}
\usepackage{amscd}
\usepackage[figuresright]{rotating}

\usepackage{color}

\usepackage{etex}
\usepackage{dsfont}
\usepackage[matrix, arrow, curve, frame, arc]{xy}
\usepackage{pgf}
\usepackage{tikz-cd}
\usetikzlibrary{arrows,shapes,automata,backgrounds,petri,calc}
\usepackage[square,sort,comma,numbers]{natbib}
 \usepackage{url}
 \usepackage{hyperref}
 \usepackage[shortlabels]{enumitem} 

\newcommand{\tikzAngleOfLine}{\tikz@AngleOfLine}
\def\tikz@AngleOfLine(#1)(#2)#3{%
\pgfmathanglebetweenpoints{%
\pgfpointanchor{#1}{center}}{%
\pgfpointanchor{#2}{center}}
\pgfmathsetmacro{#3}{\pgfmathresult}%
}

\newcommand{\bN}{\mathbb{N}}

\newcommand{\wT}{\widetilde{T}}

\newcommand{\La}{\Lambda}

\newcommand{\cF}{\mathcal{F}}

\newcommand{\cS}{\mathcal{S}}

\newcommand{\dd}{\operatorname{domdim}}

\newcommand{\spd}{{}^\diamond\!}

\newcommand{\End}{\operatorname{End}}
\newcommand{\Hom}{\operatorname{Hom}}

\newcommand{\Ext}{\operatorname{Ext}}

\newcommand{\add}{\!\operatorname{add}}

\newcommand{\m}{\!\operatorname{-mod}}

\newcommand{\proj}{\!\operatorname{-proj}}
\newcommand{\St}{\Delta}
\newcommand{\Cs}{\nabla}
\renewcommand{\L}{\Lambda}
\renewcommand{\l}{\lambda}
\newcommand{\pri}{\mathfrak{p}}

\newcommand{\mi}{\mathfrak{m}}

\newcommand{\Stsim}{\tilde{\St}}

\newcommand{\characteristic}{\operatorname{char}}
\newcommand{\Cssim}{\tilde{\Cs}}
\newcommand{\injdim}{\operatorname{idim}}
\newcommand{\domdim}{\!\operatorname{-domdim}}
\newcommand{\R}{\operatorname{R}}

\newcommand{\HN}{\operatorname{HNdim}}

\newcommand{\codomdim}{\!\operatorname{-codomdim}}
\newcommand{\inj}{\!\operatorname{-inj}}

\newcommand{\cQ}{\add_A Q}

\newtheorem{numberingthm}{Theorem}[subsection] 
\theoremstyle{definition}
\newtheorem{Def}[numberingthm]{Definition}
\newtheorem{example}[numberingthm]{Example}

\theoremstyle{plain}
\newtheorem{Prop}[numberingthm]{Proposition}
\newtheorem{Theorem}[numberingthm]{Theorem}

\newtheorem{Cor}[numberingthm]{Corollary}
\newtheorem{Lemma}[numberingthm]{Lemma}
\newtheorem{Remark}[numberingthm]{Remark}
\newtheorem{thmintroduction}{Theorem}

\theoremstyle{remark}

\begin{document}

\baselineskip=14pt

\title[Quasi-hereditary covers of Temperley-Lieb algebras]{Quasi-hereditary covers of Temperley-Lieb algebras and relative dominant dimension}

\author[T. Cruz]{Tiago Cruz}
\address[Tiago Cruz]{Institut f\"ur Algebra und Zahlentheorie,
Universit\"at Stuttgart, Germany }

\email{tiago.cruz@mathematik.uni-stuttgart.de}
\curraddr{Max-Planck-Institut f\"ur Mathematik, Vivatsgasse 7, 53111 Bonn, Germany}

\author[K. Erdmann]{Karin Erdmann}
\address[Karin Erdmann]{Mathematical Institute,
   Oxford University,
       ROQ, Oxford OX2 6GG,
   United Kingdom}
\email{erdmann@maths.ox.ac.uk}

\begin{abstract}
	Many connections and dualities in representation theory can be explained using quasi-hereditary covers in the sense of Rouquier. The concepts of relative dominant and codominant dimension with respect to a module, introduced recently by the first-named author, are important tools to evaluate and classify quasi-hereditary covers. %
	
	In this paper, we prove that the relative dominant dimension of the regular module of a quasi-hereditary algebra with a simple preserving duality with respect to a summand $Q$ of a characteristic tilting module  equals twice the relative dominant dimension of a characteristic tilting module with respect to $Q$. 
	
	To resolve the Temperley-Lieb algebras of infinite global dimension,  we apply this result to the class of Schur algebras $S(2, d)$ and $Q=V^{\otimes d}$ the $d$-tensor power of the 2-dimensional module and we completely determine the relative dominant dimension of the Schur algebra $S(2, d)$ with respect to $V^{\otimes d}$. The $q$-analogues of these results are also obtained.
	As a byproduct, we obtain a Hemmer-Nakano type result connecting the Ringel duals of $q$-Schur algebras and Temperley-Lieb algebras. From the point of view of Temperley-Lieb algebras, we obtain the first complete classification of their connection to their quasi-hereditary covers formed by Ringel duals of $q$-Schur algebras.
	
	These results are compatible with the integral setup, and we use them to deduce that the Ringel dual of a $q$-Schur algebra over the ring of  Laurent polynomials over the integers together with some projective module is the best quasi-hereditary cover of the  integral Temperley-Lieb algebra.
\end{abstract}

\maketitle

\section{Introduction}

The theory of quasi-hereditary covers, introduced in \cite{Rou}, gives a framework to study finite-dimensional algebras of infinite global dimension through algebras having nicer homological properties, for instance, quasi-hereditary algebras via an exact functor known as Schur functor. Quasi-hereditary covers appear naturally and are useful in algebraic Lie theory, representation theory and homological algebra. In particular, they are in the background of Auslander's correspondence \cite{zbMATH03517355} and in Iyama's proof of finiteness of representation dimension \cite{zbMATH01849919}. Further, quasi-hereditary algebras arise quite naturally in the representation theory of algebraic groups (\cite{MR961165, PS88}) and algebras of global dimension at most two are quasi-hereditary.

Schur algebras $S(n,d)$ form an important class of
quasi-hereditary algebras, they provide a link 
between polynomial representations of general linear groups
and representations of symmetric groups.
Classically, when $n\geq d$, the Schur algebra, via the Schur functor, is 
a quasi-hereditary cover of the group algebra of the symmetric
group $\cS_d$. This connection is seen as one of the versions of Schur--Weyl duality. 
	Indeed,  this formulation clarifies the connection between the representation theory of symmetric groups and the representation theory of Schur algebras, by detecting how their subcategories are related and how the Yoneda extension groups in these subcategories are related by the Schur functor (see also \cite{HN}). Further, this connection becomes stronger as the characteristic of the ground field increases. It was first observed in \cite{FK} that this behaviour is captured by the classical dominant dimension. However, not all quasi-hereditary covers can be evaluated using classical dominant dimension. 

To fix this, the first-named author introduced in \cite{Cr2} the concepts of relative dominant dimension and relative codominant dimension with respect to a module. Further, in \cite{Cr2} these homological invariants were exploited to  create new quasi-hereditary covers. With this, the link between Schur algebras and symmetric groups  can be regarded as a special case of quasi-hereditary covers of quotients of Iwahori-Hecke algebras.

 Temperley-Lieb algebras are among the algebras that can be regarded as quotients of Iwahori-Hecke algebras and they can have infinite global dimension. They were introduced in \cite{zbMATH03335816} in the context of statistical mechanics and they were popularised by Jones, in particular, they are used to define the Jones polynomial (see \cite{zbMATH03899758}). However, contrary to Iwahori-Hecke algebras no Hemmer-Nakano type result was known for Temperley-Lieb algebras up until now.
 Both classes of algebras are cellular (see for example \cite{zbMATH00871761}) and so an important property that they have in common is the existence of a simple preserving duality. 

Quasi-hereditary algebras with a simple preserving duality always have even global dimension. %
 Mazorchuk and Ovsienko have shown this fact in \cite{zbMATH02105773} by proving that the global dimension of a quasi-hereditary algebra with a simple preserving duality is exactly twice the projective dimension of the characteristic tilting module. Later, under much stronger conditions, the analog result for dominant dimension was obtained in \cite{FK} by Fang and Koenig exploiting that a faithful projective-injective module is a summand of the characteristic tilting module. 

The present paper has two aims. First, we will  establish that the relative dominant dimension of a quasi-hereditary algebra with respect to any summand of its characteristic tilting module is always twice as large  as that of the characteristic tilting module, in the case when the algebra has a simple preserving duality. In particular, this homological invariant is always even for such quasi-hereditary algebras. Further, Fang and Koenig's result can then be recovered from ours by just fixing the summand to be a projective-injective module. Therefore, we obtain an alternative approach to the classical case of dominant dimension without any further assumptions. 

The second aim is to study classes of quasi-hereditary covers of Temperley-Lieb algebras and their link with the representation theory of Temperley-Lieb algebras.  In particular, we aim to completely understand such a connection using the representation theory of $q$-Schur algebras and how good are the resolutions of Temperley-Lieb algebras by the Ringel duals of $q$-Schur algebras.

\subsection*{Questions to be addressed and setup}
To make our results precise, we need further notation. In general,  
assume that $B$ is a finite-dimensional algebra over an algebraically closed field.
A pair $(A, P)$ is a \emph{quasi-hereditary cover} of $B$ if $A$ is a quasi-hereditary algebra, $P$ is  a finitely generated projective $A$-module such that $B=\End_A(P)^{op}$, and in addition the
restriction of the associated Schur functor
$F:= \Hom_A(P, -)\colon A\m \rightarrow B\m
$ to the subcategory of finitely generated projective $A$-modules is full and faithful.

Let $\cF(\Delta)$ be the category of $A$-modules which have a filtration by standard modules. We would like 
the functor  $F$ to be faithful on $\cF(\Delta)$ and to induce isomorphisms
$${\rm Ext}_A^j(X, Y) \to {\rm Ext}^j_B(FX, FY)$$
for $X, Y$ modules in $\cF(\Delta)$. If this is the case for $0\leq j\leq i$ then $(A, P)$ is called an $i-\cF(\Delta)$ cover of $B$. The largest $n$ such that $(A, P)$ is an $n-\cF(\Delta)$ cover of $B$, is called the \emph{Hemmer-Nakano dimension} of
$\cF(\Delta)$ in \cite{FK}.
When $B$ is self-injective, Fang and Koenig showed that this dimension is controlled by the dominant dimension of a characteristic tilting module. In addition, they proved that if $B$ is a symmetric algebra and the quasi-hereditary cover admits a certain simple preserving duality, then the dominant dimension of a characteristic tilting module is exactly half of the dominant dimension of $A$.

Recently, in \cite{Cr2}, the situation was generalised to include cases where $B$ is not necessarily self-injective.
Moreover, it was proved in \cite{Cr2} that the Hemmer-Nakano dimension of $\mathcal{F}(\St)$ associated with a $0-\mathcal{F}(\St)$ cover can be determined using the relative codominant dimension of a characteristic tilting module with respect to a certain summand of the characteristic tilting module.

The concepts of relative dominant and relative codominant dimension (see the definition below in Subsection \ref{sec2.3})  and the concept of quasi-hereditary cover can be considered in an integral setup, that is, both of these concepts can be studied for Noetherian algebras which are finitely generated and projective as modules over a regular commutative Noetherian ring.
In \cite{Cr1}, methods were developed to reduce the computations of Hemmer-Nakano dimensions in the integral setup to computations of Hemmer-Nakano dimensions in the setup where the ground ring is an algebraically closed field. So, it will be enough for our purposes to concentrate our attention on the case when the coefficient ring is an algebraically closed field. 

 The  new approach to construct quasi-hereditary covers 
is \citep[Theorem 5.3.1.]{Cr2} and \citep[Theorem 8.1.5]{Cr2} when applied to Schur algebras (and $q$-Schur algebras). The novelty is
that it uses the Ringel dual of a Schur algebra, rather than a Schur algebra, 
and works for arbitrary parameters $n, d$. 

This can, in particular, be applied to the study of  Temperley-Lieb algebras. Indeed, the Temperley-Lieb algebras can be viewed as centraliser 
algebras of $S(2, d)$ in the endomorphism algebra of the tensor power 
$(K^2)^{\otimes d}$ (over a field $K$) and their $q$-analogues. Here $S(n, d)$ can be regarded as the centraliser algebra of $\cS_d$ in the endomorphism algebra of the tensor power $(K^n)^{\otimes d}$ over, where $(K^n)^{\otimes d}$ affords a module structure over $\cS_d$ by place permutation. Furthermore, $V^{\otimes d}:=(K^n)^{\otimes d}$ belongs in the additive closure of a characteristic tilting module over $S(n, d)$. Our cases of interest have a simple preserving duality, and in such a case, for this situation, we can without ambiguity interchange the concepts: relative dominant dimension and relative codominant dimension. 

Denote by $Q\domdim_A X$ the relative dominant dimension of an $A$-module $X$ with respect to $Q$.
In this context, the following questions arise:

\begin{enumerate}[(1)]
	\item  \emph{What is the value of $V^{\otimes d}\domdim_{S(2, d)} T$, where $T$ is a
	 characteristic tilting module  of the quasi-hereditary algebra $S(2, d)$?  What happens to this value when we replace a Schur algebra by a $q$-Schur algebra?}
	\item \emph{The Ringel duals of Schur algebras as well as Schur algebras have a simple preserving duality. Can we expect, like in the classical case (see \citep[Theorem 4.3.]{FK}), the equality $$V^{\otimes d}\domdim S(n, d)=2\cdot V^{\otimes d}\domdim_{S(n, d)} T$$ to hold in general?}
	\item \emph{Can we expect the quasi-hereditary cover of the Temperley-Lieb algebra constructed in \citep[Theorem 8.1.5]{Cr2} to be unique, in some meaningful way?}
\end{enumerate}

Our goal in this paper is to give answers to these three questions. \subsection*{Main results} Surprisingly, the answer to (2) is positive without using extra structure on $S(n, d)$ besides the quasi-hereditary structure and the existence of a simple preserving duality.

 \begin{thmintroduction}\label{thm:1} (see Theorem \ref{mainresult})
	Let $A$ be a quasi-hereditary algebra over a field $K$.
	Suppose that there exists a simple preserving duality
	$\spd (-) \colon A\m \to A\m$. Let 
	$T$ be the characteristic tilting module of $A$. Assume that $Q\in {\rm add}(T)$. Then 
	$$Q\domdim_A A = 2\cdot Q\domdim_A T.$$
\end{thmintroduction}
This result generalises \citep[Theorem 4.3.]{FK} and our methods give a new proof to their case without using any information on $A$ being gendo-symmetric, that is, an endomorphism algebra of a faithful module over a symmetric algebra. In particular, our result also works for dominant dimension exactly zero. Our approach exploits basic properties of relative injective dimensions, $\St$-filtration dimensions, some tools that were used to prove the main result of \citep{zbMATH02105773} and general properties  connecting relative dominant dimensions with relative codominant dimensions with respect to a fixed module. Observe that the left hand side of the equation in Theorem \ref{thm:1} is exactly the faithful dimension of $Q$ in sense of \cite{zbMATH01218841}. This means that, under these conditions, if the faithful dimension of $Q$ is greater or equal to 4, then the faithful dimension controls the Hemmer-Nakano dimension of $\mathcal{F}(\St)$ associated with a quasi-hereditary cover of the endomorphism algebra of $Q$.
Theorem \ref{thm:1} is applied to prove a more general case of Conjecture 6.2.4 of \cite{Cr}, that is, that the faithful dimension of a summand of a characteristic tilting module is an upper bound for the dominant dimension of the algebra provided that the former is greater or equal than two.

Combining techniques of Frobenius twisted tensor products with Theorem \ref{thm:1} we obtain a complete answer to (1):
 \begin{thmintroduction}\label{thm:2} (see also Theorem \ref{thm:5.0.7} for the $q$-version) Let $K$ be a field and let $A$ be the Schur algebra $S_{K}(2, d)$ and $T$ be the characteristic tilting module of $A$. Then,
	\begin{align*}
		V^{\otimes d}\domdim_A A=2\cdot V^{\otimes d}\domdim_A T=\begin{cases}
			d, & \text{ if } \characteristic K = 2 \text{ and } d \text{ is even}, \\
			+\infty, &\text{ otherwise }
		\end{cases}.
	\end{align*}
\end{thmintroduction}
The same approach can be used  for the $q$-analogue. 
In this case, the algebra $A$ is the $q$-Schur algebra  $S_{K, q}(2, d)$, which can be defined as the centraliser of 
the Hecke algebra $H_q(d)$ acting on    $V^{\otimes d}$,  again
for $\dim V=2$ and in the theorem the characteristic is replaced by the quantum characteristic.
When we have $v\in R$ such that $v^2=q$ and 
$\delta = -v-v^{-1}$, the Temperley-Lieb algebra $TL_{K,d}(\delta)$ is a
 quotient of this action. 
In both cases, the real difficulty lies in the case in which the  characteristic (resp. quantum characteristic) is two.

From Theorem \ref{thm:2} and its $q$-analogue, it follows that the Temperley-Lieb algebra $TL_{K, q}(\delta)$ is quasi-hereditary, and, in fact, it is the Ringel dual of a $q$-Schur algebra if $\delta\neq 0$ or $d$ is odd. Otherwise, from Theorem \ref{thm:2} follows the value of Hemmer-Nakano dimension of $\mathcal{F}(\St)$ associated with the quasi-hereditary cover of $TL_{K, q}(0)$ formed by the Ringel dual of a $q$-Schur algebra (see \citep[Theorem 8.1.5]{Cr2}). So, we have obtained a Hemmer-Nakano type result (see Corollary \ref{cor:6.2.4}) now between the Ringel dual of a $q$-Schur algebra and the Temperley-Lieb algebra. In particular, this generalises \citep[Theorem C (3), (4)]{zbMATH07021538}  for Temperley-Lieb algebras. In addition, the full subcategory of costandard modules over a $q$-Schur algebra is equivalent to the full subcategory of cell modules of the Temperley-Lieb algebra whenever $d$ is greater or equal to 6. 

If $d=2$, the Temperley-Lieb algebra is exactly an Iwahori-Hecke algebra, so nothing is new for this case. We obtain a positive answer to question (3) when we consider the Laurent polynomial ring over the integers as coefficient ring and $d>2$ (see Section \ref{sec7} and Corollary \ref{cor:7.2.1}). In such a case, the (integral) Schur functor $F$ induces an exact equivalence $\mathcal{F}(\Stsim)\rightarrow \mathcal{F}(F\Stsim),$ where the first category denotes the subcategory of modules admitting a filtration by direct summands of direct sums of standard modules over the Ringel dual of an integral $q$-Schur algebra. The quasi-hereditary cover of the integral Temperley-Lieb algebra formed by the Ringel dual of a $q$-Schur algebra is the unique quasi-hereditary cover which induces this exact equivalence.

We emphasize that the specialisation of Theorem \ref{thm:1} to projective-injective modules played a keyrole to determine the dominant dimension of Schur algebras of the form $S(n, d)$ with $n\geq d$ in \cite{FK} (also their $q$-analogues \cite{zbMATH07050778})  and it also gives an easier method to determine the dominant dimension of the blocks of the BGG category $\mathcal{O}$. It is our expectation that its use will be crucial to determine, in particular, $V^{\otimes d}\domdim S(n, d)$ and $\dd S(n, d)$ also in the cases $2<n<d$ while the latter is also an open problem for $n=2$. 

\medskip
The article is organised as follows: In Section \ref{Preliminaries}, we introduce the notation and the main properties of relative dominant dimension with respect to a module, split quasi-hereditary algebras with a simple preserving duality and cover theory to be used throughout the paper. In Section \ref{The main result}, we discuss elementary results on relative injective dimensions and we give the proof of Theorem \ref{thm:1}. We then deduce that the dominant dimension is a lower bound for the faithful dimension of a summand of a characteristic tilting module fixed by a simple preserving duality provided the latter is at least two (see Proposition \ref{Prop323}). In Section \ref{Input from  Schur algebras}, we collect results on the quasi-hereditary structure of Schur algebras $S(2, d)$, in particular, reduction techniques and how to construct partial tilting and standard modules inductively using the Frobenius twist functor. In Section \ref{dominant dimension of the regular module}, we compute the relative dominant dimension of $S(2, d)$ with respect to $V^{\otimes d}$ in terms of $V^{\otimes d}\domdim_{S(2, d)} T$, where $T$ is a characteristic tilting module of $S(2, d)$. In particular, we give the proof of Theorem \ref{thm:2} and its $q$-analogue (see Theorem \ref{thm:5.0.7}).
In Section \ref{sec6}, we recall that all Temperley-Lieb algebras can be realised as the centraliser algebras of $q$-Schur algebras in the endomorphism algebra of the tensor power $V^{\otimes d}$. As a consequence, we determine the value of Hemmer-Nakano dimension of $\mathcal{F}(\St)$ in all cases associated to the cover of the Temperley-Lieb algebra formed by the Ringel dual of a $q$-Schur algebra. This computation is contained in Corollary \ref{cor:6.2.4}. In Section \ref{sec7}, we determine the Hemmer-Nakano dimension of the above mentioned quasi-hereditary cover in the integral setup, dividing the study into two cases: the coefficient ring having or not a property of being $2$-partially $q$-divisible (see Subsection \ref{sec7.1}). When the coefficient ring does not have such property, we show that a quasi-hereditary cover with such coefficient ring has better properties. We conclude by addressing the problem of the uniqueness of this cover (see Subsection \ref{Uniqueness}).

\section{Preliminaries}\label{Preliminaries}

\subsection{The  setting}

This follows \cite{Cr2}. 
Throughout we fix a Noetherian commutative ring $R$ with identity, and 
$A$ is an $R$-algebra which is finitely generated
and projective as an $R$-module. We refer to $A$ as a projective Noetherian 
$R$-algebra. The set of invertible elements of $R$ is denoted by $R^\times$.

We denote by $A\m$ the category of finitely generated $A$-modules. 
Given $M\in A\m$, we denote by $\add_A M$ (or just $\add M$) the full subcategory
of $A\m$ whose modules are direct summands of a finite direct
sum of copies of $M$. We also denote $\add A$ by $A\proj$. 

The endomorphism algebra of a  module $M\in A\m$ is denoted by ${\rm End}_A(M)$.
We denote by $D_R$ or just $D$ the standard duality functor 
${\rm Hom}_R(-, R): A\m \to A^{op}\m$ where $A^{op}$ is the opposite algebra of $A$.

A module $M \in A\m\cap R\proj$ is said to be
\emph{$(A, R)$-injective} if it belongs to ${\rm add} DA$, and we
write $(A,R)\inj\cap R\proj$ for the full subcategory of
$A\m\cap R\proj$ whose modules are $(A, R)$-injective.

Furthermore, an exact sequence of $A$-modules which is split
as an exact sequence of $R$-modules is said to be \emph{$(A, R)$-exact}. 
In particular, an $(A, R)$-monomorphism is a homomorphism $f: M\to N$ 
that fits into an $(A,R)$-exact sequence
$0\to M\stackrel{f}\to N$. 

Given a left exact covariant additive functor $G$, we say that $X$ is a \emph{$G$-acyclic object} if $\R^{i>0}G(X)=0$. An exact sequence $0\rightarrow L\rightarrow X_0\rightarrow X_1\rightarrow \cdots$ is called a \emph{$G$-acyclic coresolution} of $L$ if all objects $X_0, X_1, \cdots$ are $G$-acyclic. Given $X\in A\m\cap R\proj$, we denote by $X^\perp$ the full subcategory $${\{M\in A\m\cap R\proj\colon \Ext_A^{i>0}(Z, M)=0, \forall Z\in \add X \}},$$ and by ${}^\perp X$ the full subcategory $\{M\in A\m\cap R\proj\colon \Ext_A^{i>0}(M, Z)=0,  \forall Z\in \add X \}$.

\subsection{Basics on approximations}

We recall definitions and some general properties relevant to approximations.
Assume  that $A$ is an $R$-algebra as above, 
and $Q$ is a fixed module in $A\m\cap R\proj$.

An $A$-homomorphism $f: M\to N$ is a \emph{left $\cQ$-approximation} of $M$ 
provided that  $N$ belongs to $\cQ$, and moreover the induced map
$${\rm Hom}_A(N, X) \to {\rm Hom}_A(M, X)
$$
is surjective for every $X\in \cQ$. 
 Dually one defines right
$\cQ$-approximations. 
Note that every module $M \in A\m$ has a left and a right $\cQ$-approximation.

\subsection{Relative (co)dominant dimension with respect to a module}\label{sec2.3}

We recall from \cite{Cr2} the definition of relative (co)dominant dimensions.

Let $Q, X\in A\m\cap R\proj$. If $X$ does not admit a 
left $\add Q$-approximation which is an $(A, R)$-monomorphism then the relative dominant dimension of
$X$ with respect to $Q$ is zero. Otherwise, \emph{the relative dominant dimension of $X$ with respect to $Q$}, denoted by
$Q\domdim_{(A, R)}X$, or $Q\domdim_A X$ when $R$ is a field, 
is the supremum of all $n\in \bN$ such that there is an $(A, R)$-exact sequence
$$0\to X\to Q_1\to Q_2 \to \ldots \to Q_n$$
with all $Q_i\in \add Q$, which remains exact under ${\rm Hom}_A(-, Q)$.

Dually one defines the \emph{relative codominant dimension}, denoted by $Q-{\rm codomdim}_{(A, R)}(X)$ with
 $Q, X$ as above:  if $X$ does not admit a surjective right $\add Q$-approximation,
then \linebreak${Q\codomdim_{(A, R)}(X)} = 0$. Otherwise it is the supremum of all $n\in \bN$ such that there is an $(A, R)$-exact sequence
$$Q_n\to Q_{n-1}\to\ldots \to Q_1\to X\to 0$$
with all $Q_i\in \add Q$, which remains exact under ${\rm Hom}_A(Q, -)$. 

Hence, $Q\codomdim_{(A, R)} X=DQ\domdim_{(A^{op}, R)} DX$. By $Q\domdim {(A, R)}$ we mean the value $Q\domdim_{(A, R)} A$.
 We will write $Q\codomdim_A X$ to denote $Q\codomdim_{(A, R)} X$ when $R$ is a field. 

The following gives a criterion towards finding $Q\domdim_{(A, R)} M$ for a given module $M$ in $A\m\cap R\proj$.

\begin{Lemma}\label{lem:2.3.1} Assume $M \in A\m\cap R\proj$, and let
	$Q_i \in \add Q$. An exact sequence 
	$$0\to M \xrightarrow{\alpha_0} Q_0 \stackrel{\alpha_1}\to Q_1\to \ldots \to Q_t$$
	remains exact under ${\rm Hom}_A(-, Q)$ if and only if for every 
	factorisation $Q_i\to {\rm im} \alpha_{i+1} \to Q_{i+1}$
	of $\alpha_{i+1}$, the $(A,R)-$monomorphism ${\rm im}\alpha_{i+1}\to Q_{i+1}$ and $\alpha_0$ are left $\add Q$-approximations.
\end{Lemma}
\begin{proof}
	See \citep[Lemma 2.1.4.]{Cr2}.
\end{proof}

In addition to the assumptions on $R, A$ and $Q$, in the following,
we also assume that 
$DQ\otimes_AQ \in R\proj$. 

It is crucial to compare relative dominant dimensions for end terms of
a short exact sequence which remains exact under ${\rm Hom}_A(-, Q)$.
This is completely described in \citep[Lemma 3.1.7]{Cr2},  for convenience, we
recall part of this.

\begin{Lemma}\label{lem:2.3.2} Let
	$M\in A\m\cap R\proj$ and consider an $(A,R)-$exact sequence
	$$0\to M_1 \to M \to M_2\to 0$$
	which remains exact under ${\rm Hom}_A(-, Q)$. Let $n=Q\domdim_AM$ and
	$n_i= Q\domdim_AM_i$ for $i=1, 2$, then:
	\begin{enumerate}[(a)]
		\item $n\geq {\rm min}\{ n_1, n_2\}$.
		\item  If $n=\infty$ and $n_1 < \infty$ then $n_2 = n_1-1$. 
	\end{enumerate}
\end{Lemma}

\begin{Cor} \label{cor:2.3.3} Let $M_i$ for $i\in I$ be a finite set of modules in $A\m\cap R\proj$. Then
	$$Q\domdim_A\left( \bigoplus_{i\in I} M_i\right) = {\rm inf} \{ Q\domdim_AM_i\mid i\in I\}.$$
\end{Cor}
\begin{proof}
	See \citep[Lemma 3.1.8]{Cr2}.
\end{proof}

Recall ${}^{\perp}Q = \{ M \in A\m\cap R\proj \mid {\rm Ext}_A^{i>0}(M, Q)=0\}$. 
The following is proved in \citep[Proposition 3.1.11.]{Cr2}.

\begin{Prop}\label{prop:2.3.4} Assume ${\rm Ext}^{i>0}_A(Q, Q)=0$, and
	$M \in {}^{\perp}Q$. An exact sequence 
	$$0\to M\to Q_1\to \ldots \to Q_n$$
	yields $Q\domdim_{(A, R)}(M)\geq n$ if and only if $Q_i \in \add Q$ and the
	cokernel of $Q_{n-1}\to Q_n$ belongs to ${}^{\perp}Q$.
\end{Prop}

The following application of Lemma \ref{lem:2.3.2} will be useful later.

\begin{Cor}\label{ddusingcoresolutions} Assume $Q\in {}^\perp Q$. Let $M\in A\m\cap R\proj$, and consider an $(A, R)$-exact
	sequence
	$$0\to M \to Q_1 \to \ldots \to Q_t \to X\to 0$$
	with $Q_i\in \add Q$. If ${\rm Ext}_A^i(X, Q)=0$ for $1\leq i\leq t$, then 
	\begin{align*}
		Q\domdim_{(A, R)}M = t + Q\domdim_{(A, R)}X.
	\end{align*}
\end{Cor}
\begin{proof}
	See \citep[Corollary 3.1.12.]{Cr2}.
\end{proof}

\subsection{Split quasi-hereditary algebras with duality}

For the definition and general properties of split quasi-hereditary
algebras we refer to \citep{CPS, Rou, cruz2021cellular, Cr1, Cr2}.  In particular, we follow the notation of \cite{cruz2021cellular, Cr1, Cr2}. One of the advantages to use such setup stems from the fact that split quasi-hereditary $R$-algebras $(A, \{\Delta(\lambda)_{\lambda\in \Lambda}\})$ are exactly the algebras so that $(S\otimes_R A, \{S\otimes_R \Delta(\lambda)_{\lambda\in \Lambda}\})$ are quasi-hereditary algebras for every commutative Noetherian ring $S$ which is an $R$-algebra. Concerning the terminology, we remark the word split arises from the endomorphism algebra $\End_A(\St(\l))$ being isomorphic to the ground ring $R$. As it was observed in \cite{Rou}, when $(A, \{\Delta(\lambda)_{\lambda\in \Lambda}\})$ is a split quasi-hereditary $R$-algebra, the objects $T(\l)$ satisfying $\add \bigoplus_{\l\in \L} T(\l)=\mathcal{F}(\Stsim)\cap \mathcal{F}(\Cssim)$ are no longer unique, in contrast to quasi-hereditary algebras over a field. For this reason, we will say that $T$ is a characteristic tilting module of $A$ if $\add T=\mathcal{F}(\Stsim)\cap \mathcal{F}(\Cssim)$ and $T$ is the (basic) characteristic tilting module of $A$ if $A$ is a quasi-hereditary algebra over a field and $T=\bigoplus_{\l\in \L} T(\l)$.

The following prepares the ground for quasi-hereditary covers, constructed from the Ringel dual $R(A)= {\rm End}_A(T)^{op}$ of a quasi-hereditary algebra $A$ with a characteristic tilting module $T$.  To see that Ringel duality is well defined in the integral setup, we refer  to \citep[Subsection 2.2.3]{Cr2}.

\begin{Prop}\label{relativedomdimtorelativecodomdim}
	Let $(A, \{\Delta(\lambda)_{\lambda\in \Lambda}\})$ be a split quasi-hereditary $R$-algebra with a characteristic tilting module $T$. Denote by $R(A)$ the Ringel dual $\End_A(T)^{op}$ of $A$.
	Suppose that $Q\in \add_A T$ is a partial tilting module. Then, 
	\begin{enumerate}[(i)]
		\item $\Hom_A(T, Q)\codomdim_{(R(A), R)} DT=Q\domdim_{(A, R)} T$.
		\item $DQ\domdim{(A, R)} = Q\codomdim_{(A, R)} DA= Q\domdim {(A,R)}.$
	\end{enumerate}
\end{Prop}
\begin{proof}
	For (i), see \citep[Proposition 6.1.1]{Cr2}. For (ii), see \citep[Corollary 3.1.5]{Cr2}.
\end{proof} Recall that $\Hom_A(T, DA)\simeq DT$ is a characteristic tilting module over $R(A)$.

 \begin{Prop}\label{dominantandcodominantoftiltingcomparison}
        Let $(A, \{\Delta(\lambda)_{\lambda\in \Lambda}\})$ be a split quasi-hereditary algebra over a field $k$.
        Assume that there exists a simple preserving duality $\spd (-)\colon A\m\rightarrow A\m$. Let $T$ be the characteristic tilting module of $A$ and assume that $Q\in \add_A T$.
 Then, \begin{enumerate}[(i)]
                \item $\spd \St(\l)\simeq \Cs(\l)$ for all $\l\in \L$;
        \item $\spd T(\l)\simeq T(\l)$ for all indecomposable modules of $T$;
        \item $Q\domdim_{(A, R)} T=Q\codomdim_{(A, R)} T$.
 \end{enumerate}
\end{Prop}
\begin{proof}(i) and (ii) follow by applying the simple preserving duality to the canonical exact sequences defining $\St(\l)$ and $T(\l)$, respectively.
        For (iii), see \citep[Proposition 3.1.6]{Cr2}.
\end{proof}

Let $(A, \{\Delta(\lambda)_{\lambda\in \Lambda}\})$ be a split quasi-hereditary algebra over a field $k$.
The \emph{$\Cs$-filtration dimension} of $X$, denoted by $\dim_{\mathcal{F}(\Cs)} X$,  is the minimal $n\geq 0$ such that 
there exists an exact sequence
$$0\rightarrow X\rightarrow M_0\rightarrow \cdots \rightarrow  M_n\rightarrow 0$$
with $M_0, \ldots, M_n \in \mathcal{F}(\Cs)$. 
Analogously, the $\St$-filtration dimension is defined. The $\Cs$-filtration dimensions first appeared in \cite{zbMATH03968901} in the study of cohomology of algebraic groups.

$\Cs$ and $\St$-filtration dimensions play a crucial role in \cite{zbMATH02116171} and \cite{zbMATH02105773} establishing that the global dimension of a quasi-hereditary algebra having a simple preserving duality is always an even number.
For us, they are of importance due to the following result.
\begin{Prop}\label{filtrationdimensionext}
                Let $(A, \{\Delta(\lambda)_{\lambda\in \Lambda}\})$ be a split quasi-hereditary algebra over a field $k$.
        Assume that there exists a simple preserving duality $\spd (-)\colon A\m\rightarrow A\m$.  If $M\in A\m$ satisfying $\dim_{\mathcal{F}(\St)} M=t<+\infty$, then $\Ext_A^{2t}(M, \spd M)\neq 0$.
\end{Prop}
\begin{proof}
        See \citep[Corollary 6]{zbMATH02105773}.
\end{proof}

\subsection{Cover theory}

The concept of a cover, and in particular, of a split quasi-hereditary cover was introduced in \cite{Rou} to give an abstract framework to connections in representation theory like Schur--Weyl duality. Given a split quasi-hereditary algebra $(A, \{\Delta(\lambda)_{\lambda\in \Lambda}\})$ over a commutative Noetherian ring $R$ and a finitely generated projective $A$-module $P$, let $B:= {\rm End}_A(P)^{op}$. We say that $(A, P)$ is a \emph{split quasi-hereditary cover} of 
$B$ if the restriction of the functor $F:=\Hom_A(P, -)\colon A\m\rightarrow B\m$,  known as Schur functor, to $A\proj$ is fully faithful.
Given, in addition, $i\in \mathbb{N}\cup\{-1, 0, +\infty \}$, following the notation of \cite{Cr1}, we say that $(A, P)$ is an $i-\mathcal{F}(\Stsim)$ (quasi-hereditary) cover of $B$ if the following conditions hold:
\begin{itemize}
	\item $(A, P)$ is a split quasi-hereditary cover of $\End_A(P)^{op}$;
	\item  The restriction of $F$  to $\mathcal{F}(\Stsim)$ is faithful;
	\item The Schur functor $F$ induces bijections $\Ext_A^j(M, N)\simeq \Ext^j_{B}(FM, FN)$, for every $M, N\in \mathcal{F}(\Stsim)$ and $0\leq j\leq i$;
\end{itemize}

Here $\mathcal{F}(\Stsim)$ denotes the resolving subcategory of $A\m\cap R\proj$ whose modules admit a finite filtration into direct summands of direct sums of standard modules $\St(\l)$, $\l\in \L$.

The optimal value of the quality of a cover is known as the Hemmer-Nakano dimension. More precisely, if $(A, P)$ is a $(-1)-\mathcal{F}(\Stsim)$ (quasi-hereditary) cover of $B$, the Hemmer-Nakano dimension of $\mathcal{F}(\Stsim)$ with respect to $F$ is $i\in \mathbb{N}\cup\{-1, 0, +\infty \}$ if $(A, P)$ is an $i-\mathcal{F}(\Stsim)$ (quasi-hereditary) cover of $\End_A(P)^{op}$ but $(A, P)$ is not an $(i+1)-\mathcal{F}(\Stsim)$ (quasi-hereditary) cover of $B$. The Hemmer-Nakano dimension of $\mathcal{F}(\Stsim)$ is denoted by $\HN_{F}(\mathcal{F}(\Stsim))$.

Major tools to compute Hemmer-Nakano dimensions are classical dominant dimension and relative dominant dimensions. This idea can be traced back to \cite{FK} which was later amplified in several directions in \cite{Cr1} and in \cite{Cr2}. This principle is briefly summarized in the following result proved in \citep[Theorem 5.3.1., Corollary 5.3.4.]{Cr2}.
Note that ${\rm Hom}_A(T, Q)$ is projective as a $B$-module.

\begin{Theorem}\label{thm:2.5.1}
		Let $R$ be a commutative Noetherian ring. Let $(A, \{\Delta(\lambda)_{\lambda\in \Lambda}\})$ be a split quasi-hereditary $R$-algebra with a characteristic tilting module $T$. Denote by $R(A)$ the Ringel dual $\End_A(T)^{op}$ of $A$.
	Assume that $Q\in \add T$ is a (partial) tilting module of $A$. Then, the following assertions hold.
	\begin{enumerate}[(a)]
		\item 	If $Q\codomdim_{(A, R)} T \geq n\geq 2$, then $(R(A), \Hom_A(T, Q))$ is an $(n-2)$-$\mathcal{F}(\Stsim_{R(A)})$ split quasi-hereditary cover of $\End_A(Q)^{op}$.
		\item  Assume, in addition, that $R$ is a field. Then,  $Q\codomdim_{(A, R)} T \geq n\geq 2$ if and only if $(R(A), \Hom_A(T, Q))$ is an $(n-2)$-$\mathcal{F}(\Stsim_{R(A)})$ split quasi-hereditary cover of $\End_A(Q)^{op}$.
	\end{enumerate}
\end{Theorem}

Let $B$ be a projective Noetherian $R$-algebra, $(A, P)$ be an $(-1)-\mathcal{F}(\Stsim)$ (quasi-hereditary) cover of $B$ and $(A', P')$ be an $(-1)-\mathcal{F}(\Stsim')$ (quasi-hereditary) cover of $B$. We say that $(A, P)$ is equivalent to $(A', P')$ as quasi-hereditary covers if there exists an equivalence functor $H\colon A\m\rightarrow A'\m$ which restricts to an equivalence of categories between $\mathcal{F}(\Stsim)$ and $\mathcal{F}(\Stsim')$ making the following diagram commutative
\begin{equation*}
	\begin{tikzcd}
		A\m \arrow[d, "H"]  \arrow[r, "\Hom_A(P{,} -)", outer sep=0.75ex] &  B\m \arrow[d, "L"] \\
		A'\m  \arrow[r, "\Hom_{A'}(P'{,} -)", outer sep=0.75ex]  &B\m
	\end{tikzcd},
\end{equation*}
for some equivalence of categories $L$. The first application of uniqueness of covers goes back to \cite{Rou}. Split quasi-hereditary covers with higher values of Hemmer-Nakano dimension associated to them are essentially unique. In fact, this is due to the following result which can be found in \citep[Corollary 4.3.6.]{Cr1}.
\begin{Cor}\label{cor:2.5.2}
	Let $B$ be a projective Noetherian $R$-algebra, $(A, P)$ be a $1-\mathcal{F}(\Stsim)$ (quasi-hereditary) cover of $B$ and $(A', P')$ be a $1-\mathcal{F}(\Stsim')$ (quasi-hereditary) cover of $B$.  If there exists an exact equivalence $L\colon B\m\rightarrow B\m$ which restricts to an exact equivalence between $\mathcal{F}(\Hom_A(P, -)\Stsim)$ and $\mathcal{F}(\Hom_{A'}(P', -)\Stsim')$, then $(A, P)$ is equivalent as split quasi-hereditary cover to $(A', P')$.
\end{Cor}

For a more detailed exposition on cover theory and Hemmer-Nakano dimensions we refer to \cite{Cr1, Cr2}.

\section{The main result}\label{The main result}

The aim of this section is to prove Theorem \ref{thm:1}.  

\subsection{Relative injective dimension}  The following concept of relative injective dimension will be useful as a tool in the proof of Theorem \ref{thm:1}.

\begin{Def} 
	Let $\mathcal{A}$ be a full subcategory of $A\m$. We define the \emph{$\mathcal{A}$-injective dimension} of $N\in A\m$ (or the \emph{relative injective dimension of $N$ with respect to $\mathcal{A}$}) as the value \begin{align}
		\inf\{n\in \mathbb{N}\cup \{0\}\colon \Ext_A^{i>n}(M, N)=0, \forall M\in \mathcal{A} \}.
	\end{align}
	We denote by $\injdim_\mathcal{A} N$ the $\mathcal{A}$-injective dimension of $N$.
	Analogously, we define the \emph{$\mathcal{A}$-projective dimension} of $N\in A\m$ as the value
	\begin{align}
		\inf\{n\in \mathbb{N}\cup \{0\}\colon \Ext_A^{i>n}(N, M)=0, \forall M\in \mathcal{A} \}.
	\end{align}
\end{Def}

\begin{Lemma}
	Let $A$ be a projective Noetherian $R$-algebra and let $Q\in A\m$ satisfying $\Ext_A^{i>0}(Q, Q)=0$.
	Then, the following assertions hold.\label{lemmamainresultthree}
	\begin{enumerate}
		\item If there exists an exact sequence $0\rightarrow X\rightarrow Y\rightarrow Z\rightarrow 0$ with $Y\in \add Q$, then $$\injdim_{{}^\perp Q} X\leq 1+\injdim_{{}^\perp Q} Z.$$
		\item If there exists an exact sequence $0\rightarrow X\rightarrow X_r\rightarrow \cdots \rightarrow X_1 \rightarrow Z\rightarrow 0$ with $X_1, \ldots, X_r\in \add Q$, then $\injdim_{{}^\perp Q} X\leq r+\injdim_{{}^\perp Q} Z$.
		\item If there exists an exact sequence $0\rightarrow X\rightarrow X_r\rightarrow \cdots \rightarrow X_1 \rightarrow Z\rightarrow 0$ with $X_1, \ldots, X_r\in \add Q$, then for every $Y\in Q^{\perp}$, $\Ext_A^i(X, Y)\simeq \Ext_A^{i+r}(Z,Y)$ for all $i\in \mathbb{N}$.
	\end{enumerate}
\end{Lemma}

\begin{proof}
	For each $M\in {}^\perp Q$, applying $\Hom_A(M, -)$ yields that $\Ext_A^i(M, Z)\simeq \Ext_A^{i+1}(M, X)$ for all $i\geq 1$. Hence, (i) follows. By induction and using (i), (ii) follows.
	Denote by $C_i$ the image of $X_{i+1}\rightarrow X_i$ for all $i=1, \ldots, r-1$. By applying $\Hom_A(-, Y)$ we deduce that $\Ext_A^i(X, Y)\simeq \Ext_A^{i+1}(C_{r-1}, Y)\simeq \Ext_A^{i+2}(C_{r-2}, Y)\simeq \Ext_A^{i+r-1}(C_1, Y)\simeq \Ext_A^{i+r}(Z, Y)$.
\end{proof}

\subsection{Computing relative dominant dimension of the regular module using a characteristic tilting module}

In general, the relative codominant dimension of a characteristic tilting module with respect to a partial tilting module gives a lower bound to the relative dominant dimension of the regular module with respect to a partial tilting module (see \citep[Theorem 5.3.1(a)]{Cr2}). In the following, we will see that this lower bound can be sharpened using also the relative dominant dimension of a characteristic tilting module with respect to a partial tilting module. 
\begin{Lemma}\label{mainresultlemma}
Let $(A, \{\Delta(\lambda)_{\lambda\in \Lambda}\})$ be a split quasi-hereditary $R$-algebra with a characteristic tilting module $T$. Denote by $R(A)$ the Ringel dual $\End_A(T)^{op}$ of $A$.
 Suppose that $Q\in \add T$. Then,
        \begin{align}
                Q\domdim (A, R) \geq  Q\domdim_{(A, R)} T + Q\codomdim_{(A, R)} T.
        \end{align}
\end{Lemma}
\begin{proof}
        Observe that $DQ\otimes_A Q\in R\proj$ (see for example \citep[A.4.3.]{Cr1}).
        By \citep[Theorem 5.3.1(a)]{Cr2} and \cite[Corollary 3.1.5]{Cr2}, we obtain that \begin{align}
        Q\domdim_{(A, R)} A=Q\codomdim_{(A, R)} DA\geq Q\codomdim_{(A, R)} T.\end{align}
        If $Q\domdim_A T=0$, then there is nothing more to prove. Assume that $n:=$ \linebreak${Q\domdim_{(A, R)} T}\geq 1$.
        By Proposition \ref{relativedomdimtorelativecodomdim}(i), $\Hom_A(T, Q)\codomdim_{R(A)} DT=n$. Then there exists an exact sequence
        \begin{align}
                0\rightarrow C\rightarrow X_n\rightarrow \cdots \rightarrow X_1\rightarrow DT\rightarrow 0 \label{eqmain2}
        \end{align}with all $X_i\in \add \Hom_A(T, Q)$, and so they are
	projective modules over $R(A)$. The subcategory $\mathcal{F}(\Stsim_{R(A)})$ is closed under kernels of epimorphisms and since $DT$ is a characteristic tilting module over $R(A)$ we obtain that $C\in \mathcal{F}(\Stsim_{R(A)})$. Thus, (\ref{eqmain2}) remains exact under $T\otimes_{R(A)}$ which is left adjoint to  $\Hom_A(T, -)$, and we obtain an exact sequence
\begin{align}
        0\rightarrow \overline{C}\rightarrow \overline{X_n}\rightarrow \cdots\rightarrow \overline{X_1}\rightarrow T\otimes_{R(A)} DT\rightarrow 0 \label{eqmain3}
\end{align}with all $\overline{X_i}\in \add T\otimes_{R(A)} \Hom_A(T, Q)=\add Q$ since $Q\in \add T$. Moreover, $T\otimes_{R(A)} DT\in R\proj$ by \citep[A.4.3.]{Cr1} and so $$T\otimes_{R(A)} DT\simeq T\otimes_{R(A)} \Hom_A(T, DA)\simeq DA$$ since ${DT\domdim_{(A, R)} DA}=+\infty$ (see \citep[Theorem 3.1.1]{Cr2}). By construction, $\overline{C}\in \mathcal{F}(\Cssim)$ and so (\ref{eqmain3}) remains exact under $\Hom_A(Q, -)$. By the dual version of \citep[Corollary 3.1.12]{Cr2}, we obtain that $${Q\codomdim_{(A, R)} DA}=n+Q\codomdim_{(A, R)} \overline{C}.$$
By \citep[Theorem 5.3.1(a)]{Cr2}, ${Q\codomdim_{(A, R)} \overline{C}}\geq Q\codomdim_{(A, R)} T.$
\end{proof}

Surprisingly, the following result generalises \citep[Theorem 4.3]{FK} without using any techniques on symmetric algebras. In particular, for this proof we do not use the fact that the endomorphism algebra of a faithful projective-injective module over a quasi-hereditary algebra with a simple preserving duality is a symmetric algebra.

\begin{Theorem}\label{mainresult}
        Let $(A, \{\Delta(\lambda)_{\lambda\in \Lambda}\})$ be a split quasi-hereditary algebra over a field $k$. Suppose that there exists a simple preserving duality $\spd(-)\colon A\m\rightarrow A\m$. Let $T$ be a characteristic tilting module of $A$. Assume that $Q\in \add T$. Then,
        \begin{align}
                Q\domdim A=Q\codomdim_A DA= 2\cdot Q\codomdim_A T=2\cdot Q\domdim_A T.
        \end{align}
\end{Theorem}

\begin{proof}
        By \citep[Corollary 3.1.5]{Cr2}, \citep[Proposition 3.1.6]{Cr2} and Lemma \ref{mainresultlemma}, it remains to show that $Q\codomdim_A DA\leq 2 \cdot Q\codomdim_A T$. If $Q\codomdim_A T=+\infty$ then there is nothing to prove. Denote by $n$ the value $Q\domdim_A T=Q\codomdim_A T$.   Assume first that $n>0$. 
So we can consider again exact sequences of the form (\ref{eqmain2}) and (\ref{eqmain3}). Assume, for a contradiction, that $Q\codomdim_A DA>2n$. Hence, also $Q\codomdim_A \overline{C}>n$ according \citep[dual of Corollary 3.1.12]{Cr2}. So there exists an exact sequence
        \begin{align}
                0\rightarrow L\rightarrow \overline{X_{2n+1}}\rightarrow \overline{X_{2n}}\rightarrow \cdots\rightarrow \overline{X_{n+1}}\rightarrow \overline{C}\rightarrow 0
        \end{align}
which remains exact under $\Hom_A(Q, -)$ and $\overline{X_i}\in \add Q$, $i=n+1, \ldots, 2n+1$. In particular, $0\rightarrow L\rightarrow \overline{X_{2n+1}}\rightarrow \cdots\rightarrow \overline{X_1}\rightarrow DA\rightarrow 0$
is an $\Hom_A(Q, -)$-acyclic coresolution of $L$, so it can be used to compute $\Ext_A^i(Q, L)$ for all $i$.
Since it remains exact under $\Hom_A(Q, -)$ we obtain that $\Ext_A^{i>0}(Q, L)=0$ and so $L\in Q^\perp$ and $\spd L\in {}^\perp Q$.
Let $D$ be the kernel of the map $\overline{X_{n+1}}\rightarrow \overline{C}$ and consider the exact sequence \begin{align}
        0\rightarrow D\rightarrow \overline{X_{n+1}}\rightarrow \overline{X_n}\rightarrow \cdots\rightarrow  \overline{X_1}\rightarrow DA\rightarrow 0. \label{eqmain9}
\end{align} By Lemma \ref{lemmamainresultthree}(2),  $\injdim_{{}^\perp Q} D\leq n+1$ and since (\ref{eqmain9}) remains exact under $\Hom_A(Q, -)$ we have that $D\in Q^{\perp}$. On the other hand, observe that $D$ cannot belong to $\mathcal{F}(\Cs)$ because otherwise (\ref{eqmain9}) would remain exact under $\Hom_A(T, -)$ yielding that $n<\Hom_A(T, Q)\codomdim_{R(A)} DT$ contradicting the definition of $n$.

So the exact sequence $0\rightarrow D\rightarrow \overline{X_{n+1}}\rightarrow \overline{C}\rightarrow 0$ yields that $\dim_{\mathcal{F}(\Cs)}D=1$. Hence, $\dim_{\mathcal{F}(\St)}\spd D=1$. By Corollary 6 of \cite{zbMATH02105773} we obtain that $0\neq \Ext_A^2(\spd D, \spd \spd D)\simeq \Ext_A^2(\spd D, D)$. By Lemma \ref{lemmamainresultthree}(3) on the exact sequence $0\rightarrow \spd D\rightarrow \spd \overline{X_{n+2}}\rightarrow\cdots\rightarrow \spd \overline{X_{2n+1}}\rightarrow \spd L\rightarrow 0$ we obtain $0\neq \Ext_A^2(\spd D, D)\simeq \Ext_A^{n+2}(\spd L, D)$. This contradicts $\injdim_{{}^\perp Q} D$ being at most $n+1$.
       We will now treat the case $n=0$.  Assume, for sake of contradiction, that $Q\codomdim_A DA\geq 1$, then there exists an exact sequence $0\rightarrow L\rightarrow X_1\rightarrow DA\rightarrow 0$ which remains exact under $\Hom_A(Q, -)$ and $X_1\in \add Q$. Hence, $L\in Q^{\perp}$, $\spd L\in {}^\perp Q$, $\dim_{\mathcal{F}(\Cs)}(L)\leq 1$ and the ${}^\perp Q$-injective dimension of $L$ is at most one. In particular, $\Ext_A^2(\spd L, L)= 0$. By Proposition \ref{filtrationdimensionext}, we must have that $L\in \mathcal{F}(\Cs)$. But, then applying $\Hom_A(T, -)$ to   $0\rightarrow L\rightarrow X_1\rightarrow DA\rightarrow 0$ yields that $\Hom_A(T, Q)\codomdim_{R(A)} \Hom_A(T, DA)\geq 1$ which, in turn, implies that $n=Q\domdim T\geq 1$ by Proposition \ref{relativedomdimtorelativecodomdim}(i).
\end{proof}

The following will in particular  give a positive answer to the Conjecture 6.2.4 of \cite{Cr}.

\begin{Prop}\label{Prop323}
        Let $(A, \{\Delta(\lambda)_{\lambda\in \Lambda}\})$ be a split quasi-hereditary algebra over a field $k$ with a simple preserving duality. Let $T$ be a characteristic tilting module of $A$. Assume that $Q\in \add T$ satisfying $Q\domdim_A A\geq 2$. Then,
\begin{align}
        Q\domdim_A A=2\cdot Q\domdim_A T\geq 2\dd_A T=\dd A.
\end{align}
\end{Prop}
\begin{proof}Let $P$ be a faithful projective-injective module over $A$.
        By assumption, $Q\domdim_A A\geq 2$, so by \citep[Corollary 3.1.8]{Cr2} it follows that $Q\domdim_A P\geq 2$. Since $P$ is injective we must have that $P\in \add Q$. Hence, $\Hom_A(T, P)\in \add \Hom_A(T, Q)$.

  By Proposition \ref{relativedomdimtorelativecodomdim}(i), \begin{align}
                Q\domdim_A T&=\Hom_{R(A)}(T, Q)\codomdim_{R(A)} DT
                \geq \Hom_A(T, P)\codomdim_{R(A)} DT \nonumber\\
                &=P\domdim_A T=\dd_A T.
        \end{align}
Applying Theorem \ref{mainresult} to the partial tilting modules $P$ and $Q$, the result follows.
\end{proof}

\section{Input from  Schur algebras}\label{Input from  Schur algebras}

The main work to prove the second main result, to determine the Hemmer-Nakano dimension of $\mathcal{F}(\St)$ over 
the quasi-hereditary cover for the Temperley-Lieb algebra, in Sections \ref{sec6} and \ref{sec7},  
is done for Schur algebras, and we can work over an algebraically closed field.
In this section, we give an outline of the background.
To keep the notation simple, we do this for the classical case.

Assume $K$ is an algebraically closed field. 
The Schur algebra $S=S_K(n,d)$ (or just $S(n, d)$) of degree $d$ over $K$
can be defined in different ways. 
One can start with the symmetric group $\cS_d$ which acts (on the right)
by place
permutations on the tensor power $V^{\otimes d}$ where $V$ is an $n$-dimensional
vector space. Then the \emph{Schur algebra}
$S(n, d) $ is the endomorphism algebra ${\rm End}_{K\cS_d}(V^{\otimes d})$. 
Analogously, the 
\emph{integral Schur algebra} $S_R(2, d)$ is defined as the endomorphism algebra
${\rm End}_{R\cS_d}((R^2)^{\otimes d})$ where $(R^2)^{\otimes d}$ affords a right
$R\cS_d$-module structure via place permutations.
Alternatively one can construct $S(n, d)$ via the general linear group $GL(V)$, for details see for example \cite{Gr} or \cite{Do}. 
The first route shows that the endomorphism algebra of $S(n,d)$ acting on 
$V^{\otimes d}$ is a quotient of $K\cS_d$. The second approach allows one use
tensor products and Frobenius twists as tools to study representations.

The Schur algebra $S(n, d)$ is  quasi-hereditary, with respect
 to the dominance order on the set $\Lambda^+(n, d)$ of partitions of $d$ with
at most $n$ parts, which is the standard labelling set for simple modules.
It has a simple preserving duality $\spd(-)$ (see for example \citep[p.83]{Do2}).  For each partition $\l$ of $d$ with at most $n$ parts, the corresponding simple module will be denoted by $L(\lambda)$. 
We denote the standard module with simple top $L(\lambda)$ by $\Delta(\lambda)$, then the costandard module with simple socle $L(\lambda)$ is $\nabla(\lambda) = \spd\Delta(\lambda)$. 
For background we refer to 
\cite{E1} or \cite{DR}, \cite{DR1}.

Of central importance for the quasi-hereditary structure  is the characteristic tilting module $T$: By \cite{Ri} the indecomposable modules
in $\cF(\Delta)\cap \cF(\nabla)$ are in bijection with the weights. Write 
 $T(\lambda)$ for  the indecomposable labelled by 
$\lambda \in \Lambda^+(n, d)$. Then the 
direct sum $T:= \bigoplus_{\lambda\in \L^+(n, d)} T(\lambda)$ (or a module
with the same indecomposable summands) 
is a distinguished tilting module, known as the \emph{characteristic
tilting module} of $S$. 
Its endomorphism algebra $R(S):= {\rm End}_S(T)^{op}$ is again quasi-hereditary
and $R(R(S))$ is Morita equivalent (as quasi-hereditary algebra) to $S$. 

For each $\lambda$, there is an associated exact sequence
\begin{align}
	0\to \Delta(\lambda) \to T(\lambda) \to X(\lambda)\to 0 \label{(4.1)}
\end{align} 
 
 where $X(\lambda)$ has $\Delta$-filtration where only $\Delta(\mu)$ with
 $\mu< \lambda$ occur. We will refer to this as a standard sequence.

 We follow the usual practice in algebraic Lie theory to refer to 
 a module in ${\rm add}(T)$ as a tilting module, and to $T$ as a full
 tilting module (this will not be ambiguous here).

For the connection between Schur algebras and symmetric groups, the tensor
space $V^{\otimes d}$ is of central importance. 
As it happens, the tensor space is a direct sum of tilting modules, and 
$T(\lambda)$ occurs as a summand if and only if $\lambda$ is $p$-regular
(that is does not have $p$ equal parts). For the quantum case, $T(\lambda)$ occurs in the tensor space if and only if
$\lambda$ is $\ell$-regular where $q$ is a primitive $\ell$-th root of 1.  
This is proved in \citep[4.2]{E1}, or combining  the reasoning of \citep[4.2]{E1} with \citep[2.2(1), 4.3, 4.7]{Do2} respectively. Hence, the following result has become folklore.

\begin{Lemma}\label{lem:4.0.1}
	Assume that $n=2$ and $d$ is a natural number. 
 If $\characteristic K\neq 2$ or $d$ is odd, then $V^{\otimes d}$ is a characteristic tilting module over $S_K(2, d)$.
\end{Lemma}
\begin{proof}
	If $K$ has characteristic zero, then the Schur algebra $S(2, d)$ is semi-simple (see for example \citep[(2.6)e]{Gr}) and since $V^{\otimes d}$ is faithful over $S(2, d)$ it contains the regular module in its additive closure, and in particular, $V^{\otimes d}$ is a characteristic tilting module. If $K$ has positive characteristic, as discussed before $V^{\otimes d}$ is a characteristic tilting module over $S(2, d)$ if and only if all partitions of $d$ in at most $2$ parts are $\characteristic K$-regular partitions of $d$.  Of course, all partitions of $d$ in at most $2$ parts are $p$-regular if $p>2$. If $d$ is odd, then there are no partitions of $d$ in exactly two equal parts. 
\end{proof}

From now on we assume $n=2$ and $\characteristic K =2$, or in the quantum case that $\ell = 2$. We also  assume $d$ is even (unless specified differently).

\subsection{On the quasi-hereditary structure of $S(2, d)$}

Let $S=S(2, d)$, and let $e=\xi_{(d)}$ be the idempotent corresponding to the largest weight (in the notation of \cite{Gr}). Then $SeS$ is
an idempotent heredity ideal and $S/SeS$ is isomorphic to $S(2, d-2)$ (for details see for example \cite{E2}).  Since $SeS$ is a heredity ideal 
corresponding to $(d)$,  factoring it out is compatible with the quasi-hereditary structure. Furthermore, as it is proved in the appendix of \cite{DR}
computing ${\rm Ext}^i$'s for $S/SeS$-modules is the same whether in $S$ or in $S/SeS$. In particular, $S(2, d)\m$ is the full subcategory of $S(2, d+2)\m$ consisting of modules whose composition factors are different from those appearing in the top of $Se$. 

We work mostly with the restrictions of simple modules, (co)standard modules and tilting modules to $SL(2, K)$.
 Recall that
$L(\lambda)$ and $L(\mu)$ are isomorphic as $SL(2, K)$-modules if and only if  they can be regarded both as $S(2, d)$-modules for some large $d$ and they are isomorphic as $S(2, d)$-modules. This fact can be seen using the canonical surjective map of $KSL(2, K)$ onto $S(2, d)$. Since every partition $(\l_1, \l_2)$ of $d$ in at most $2$ parts  is completely determined by the value $\l_1-\l_2$, it follows that $L(\lambda)$ and $L(\mu)$ are isomorphic as $SL(2, K)$-modules if and only if $\lambda_1-\lambda_2 = \mu_1 - \mu_2$. Similarly for standard modules and tilting modules. 

We therefore label these modules by $m =\lambda_1-\lambda_2$ if $\lambda = (\lambda_1, \lambda_2)$ (such labellings can also be found for example in \citep[Subsection 3.2]{EL}).
This means that we consider Schur algebras $S=S(2,d)$, allowing degrees to vary
but keeping the parity.
We make the convention that
we view tacitly modules for $S(2, d')$ with $d'\leq d$ of the same parity as modules for $S(2, d)$.
We say that such a degree $d'$ 
is admissible for the module defined in degree $d$.
With this, the weights labelling the simple modules for $S(2, d)$ are precisely all non-negative integers $m\leq d$ of the same parity.
The dominance order when $p=2$ and the degree is even, is the linear order.

The tilting module $T(0)$ is simple, it is the trivial module for $SL(2, K)$. 
As a building block, the tilting module $T(1)$ appears,  which is isomorphic
to the natural $SL(2, K)$-module $V$. Furthermore,
$T(2) \cong V^{\otimes 2}$. For $d\geq 4$ we have that $V^{\otimes d}$ is the direct sum of $T(k)$ where all $T(k)$ occur for
$k$ of the same parity of $d$, except that $T(0)$ does not occur when $d$ is even.
(See for example \cite{E1}).

\subsection{The category $\cF(\Delta)$ and projective modules}

Non-split extensions of  standard modules satisfy a directedness property, that is
$$\Ext_S^1(\Delta(r), \Delta(s))\neq 0 \ \ \mbox{ implies} \ \ r< s.
$$
This has the following immediate consequence:

\begin{Lemma}\label{lem:4.2.1} Every module in
$\cF(\Delta)$ has a filtration in which weights of $\Delta$-quotients increase from top to bottom. 
\end{Lemma}
\begin{proof} This follows for example from \citep[Lemma 1.4]{DR}, 
	see also \citep[B.1.6]{Cr1}. See \cite{zbMATH03704845}, for an earlier reference.
\end{proof}

Of main interest for us are the indecomposable projective modules.
Let $P_d(m)$ denote the indecomposable projective of $S(2, d)$ with simple quotient $L(m)$. 
Recall $P_d(m)$ has a $\Delta$-filtration, and that the filtration multiplicities $[P_d(m):\Delta(w)]$  are the same as 
the decomposition numbers.  
That is, 
$$[P_d(m): \Delta(w)] = (\nabla(w):L(m)) = (\Delta(w):L(m)).
$$
where we write $(M:L(m))$ for the multiplicity of $L(m)$ as a composition factor of the module $M$.
Note this also shows that projective modules depend on the degree $d$. 
In this case, decomposition numbers are always $0$ or $1$, see 
\citep[Prop. 2.2, Theorem 3.2.]{H}.
We give an example in Figure \ref{fig:decomp}.

\begin{figure}[h]
        \begin{tiny}
                \setlength\arraycolsep{2pt}
                \[
                \begin{array}{r|*{4}c|*{4}c|*{8}c|*{8}cc}
                \hline
                0& 1 &&&& &&&&& &&&&& &&&&& &&&&&\\
                2& 1&1 &&& &&&&& &&&&& &&&&& &&&&&\\
                4&  1&1&1 && &&&&& &&&&& &&&&& &&&&&\\
                6&  1&.&1&1 & &&&&& &&&&& &&&&& &&&&&\\
                \hline
                8& 1&.&1&1&1 &&&&& &&&&& &&&&& &&&&&\\
                10& 1&1&1&.&1& 1&&&& &&&&& &&&&& &&&&&\\
                12& 1&1&.&.&1& 1&1&&& &&&&& &&&&& &&&&&\\
                14& 1&.&.&.&1& .&1&1&& &&&&& &&&&& &&&&&\\
                \hline
                16& 1&.&.&. &1 &. &1&1&1& &&&&& &&&&& &&&&&\\
                18& 1&1&.&.& 1& 1&1&.&1&1 &&&&& &&&&& &&&&&\\
                20& 1& 1& 1&.& 1&1 .&.&1 &1&1& 1&&&& &&&&& &&&&&\\
                22& 1&.&1& 1& 1& .&.&.&1&.& 1&1&&& &&&&& &&&&&\\
                24& 1&.&1&1&.& .&.&.&1&.& 1&1&1&& &&&&& &&&&&\\
                26& 1&1&1&.&.& .&.&.&1&1& 1&.&1&1& &&&&& &&&&&\\
                28& 1&1&.&.&.& .&.&.&1&1& .&.&1&1&1 &&&&& &&&&&\\
                 30& 1&.&.&.&.& .&.&.&1&.& .&.&1&.&1& 1&&&& &&&&&\\
                \hline 
                32& 1&.&.&.&.& .&.&.&1&.& .&.&1&.&1& 1&1&&& &&&&&\\
                34& 1&1&.&.&.& .&.&.&1&1& .&.&1&1&1& .&1&1&& &&&&&\\
                36& 1&1& 1&.&.& .&.&.&1&1& 1&.&1&1&.& .&1&1&1& &&&&&\\
                38& 1&.&1&1&.& .&.&.&1&.& 1&1&1&.&.& .&1&.&1&1 &&&&&\\
                 40& 1&.&1&1&1& .&.&.&1&.& 1&1&.&.&.& .&1&.&1&1& 1&&&&\\
                42& 1&1&1&.&1& 1&.&.&1&1& 1&.&.&.&.& .&1&1&1&.& 1&1&&&\\
                44& 1&1&.&.&1& 1&1&.&1& 1& .&.&.&.&.& .&1&1&.&.& 1&1&1&&\\
                46&1&.&.&.&1&.&1&1&1&.&.&.&.&.&.&.&1&.&.&.&1&.&1&1&\\
                                \end{array} \]
                \begin{center}
                        $\vdots$
                \end{center}
        \end{tiny}
	\caption{Decomposition matrix for $S(2, 46)$ for $p=2$. The $(m,n)$-entry denotes $(\Delta(m):L(n))$, the column label is the same as the row label.} 
  \label{fig:decomp}
\end{figure}

It follows that either $P_{d}(m) \cong P_{d-2}(m)$ as a module for $S(2, d)$, or else there is a non-split exact sequence
\begin{align}
	0\to \Delta(d) \to P_{d}(m) \to P_{d-2}(m)\to 0. \label{4.2}
\end{align}
Namely, the top of $P_{d-2}(m)$ is $L(m)$, so there
is a surjective homomorphism from $P_{d}(m)$ onto $P_{d-2}(m)$. Recall that
$\cF(\Delta)$ is closed under kernels of epimorphisms. 
By the filtration property in Lemma \ref{lem:4.2.1}
if this is not an isomorphism, then its kernel is a direct sum of copies
of $\Delta(d)$ and there is only one since the decomposition numbers
are $\leq 1$.

\subsection{Twisted tensor product methods}\label{Twisted tensor product methods}

Let $(-)^F$ denote the Frobenius twist (see \citep[page 64]{Do2}), 
this is an exact functor.
In our setting, that is for even characteristic,  we have the following tools, due to \cite{Do}.
Odd degrees when $p=2$ are less important. Namely, each block of $S(2, d)$ for $d$ odd  is Morita equivalent
to some block of some Schur algebra $S(2, x)$ with $x=(d-1)/2$ via the functor $\Delta(1)\otimes (-)^F$, see for example
\citep[Lemma 1]{EH} or \citep[Section 4, Theorem]{zbMATH00645087}.

\begin{enumerate} 
	\item \label{4dot3dot1} \begin{enumerate}[(a)]
		\item Let $m=2t$.  There is an exact sequence  of $S$-modules
		$$0\to \Delta(t-1)^F\to \Delta(m)\to \Delta(t)^F\to 0$$ 
		Taking contravariant duals gives the analog for costandard modules.
		\item Let $m=2t+1$, then
		$\Delta(m)\cong L(1)\otimes \Delta(t)^F$. 
	\end{enumerate}
			(See for example \citep[Prop. 3.3]{Cox}).
\end{enumerate}

We note that this determines recursively the decomposition numbers,
as input using that $\Delta(t)$ is simple and isomorphic to $L(t)$ for
$t=0, 1$, recall $p=2$. 
This can also be used to show that when $p=2$ and $d$ is even, the algebra
$S(2, d)$ is indecomposable. Further, this also implies, by induction, that the decomposition numbers are always $0$ and $1$ when $p=2$ and $d$ is even.

\begin{enumerate}[resume]
\item We have a complete  description of the indecomposable tilting modules
in this case. 
		 We have already described   $T(m)$ for $m\leq 2$. 
		The following is due to S. Donkin, see \citep[Example 2 p. 47]{Do}. 
\end{enumerate}

\begin{Prop}\label{prop:4.3.1} Let $m=2s$ and $m\geq 2$, then
$$T(m)\cong T(2)\otimes T(s-1)^F$$
        If $m=2s+1$, then $T(r) \cong T(1) \otimes T(s)^F$.
\end{Prop}

This describes recursively  all indecomposable tilting modules.
Note that tilting modules are not changed if the degree increases.

The following shows that  filtration multiplicities $[T(m):\Delta(w)]$ are $\leq 1$.

\begin{Prop}\label{prop:4.3.2} The $\Delta$-filtration multiplicities
	of indecomposable tilting
modules in even degree can be computed
recursively from
$$0\to \Delta(2t+2) \to T(2)\otimes \Delta(t)^F \to \Delta(2t) \to 0.$$
\end{Prop}

To prove this, one may  specialize \citep[Prop. 3.4]{Cox}.

We will see below that modules $T(2)\otimes X^F$ for $X$ in $\cF(\Delta)$ have
infinite relative dominant dimension with respect to $V^{\otimes d}$. This means that we can use 
Lemma \ref{lem:2.3.2} (from (3.1.7) of \cite{Cr2}) to relate the  relative $V^{\otimes d}-$dominant dimension of the end terms, and this suggests a route 
towards the proof of our second main result.

We define a  {\it twisted filtration} of a module  $M\in \cF(\Delta)$ to be  a  filtration
where each quotient is isomorphic to
$T(2)\otimes \Delta(t)^F$ for some $t$.

\begin{Lemma}\label{lem:4.3.3} Let $m=2s\geq 1$. Then the tilting module $T(m)$  has a twisted filtration
$$0= M_k \subset M_{k-1} \subset \ldots \subset M_1\subset M_0=T(m)$$
with $M_{i-1}/M_i \cong T(2)\otimes \Delta(s_i)^F$, with quotients
        $\Delta(2s_i)$ and $\Delta(2s_i+2)$, for $s_1 < s_2<\ldots < s_k$.
\end{Lemma}
\begin{proof}
	We have  $T(m)\cong T(2)\otimes T(s-1)^F$. The module
	$T(s-1)$ has a $\Delta$-filtration 
	$$N_k=0\subset N_{k-1} \subset \ldots \subset N_0=T(s-1)$$
	with $N_{i-1}/N_i \cong \Delta(s_i)$ and such that
	$s_1 < s_2 < \ldots < s_k$, by Lemma \ref{lem:4.2.1}. 
	Applying the exact functor $T(2)\otimes (-)^F$ gives the claim.
\end{proof}

\begin{Remark}  \label{rem:4.1} \  
	\begin{enumerate}
		\item The algebra $S(2, d)$ is Ringel self-dual for $p=2$ and $d$ even
		if and only if $d=2^{n+1}-2$ for some $n$, see \cite{EH} and for
		a functorial proof see \cite{EL}. 
		In  \cite{EH} and \cite{EL}  it was identified precisely
		which tilting modules are projective (and injective) for these degrees.
		\item Each $T(m)$ has a simple top, this is also part of \cite{EH}, \cite{EL}.
	\end{enumerate}
\end{Remark}

\section{The relative dominant dimension of the regular module with respect to $V^{\otimes d}$}\label{dominant dimension of the regular module}

Let $S=S(2, d)$ and assume that $K$ has characteristic $p$. Recall that 
the indecomposable summands of  $V^{\otimes d}$ are precisely the $T(\lambda)$ where $\lambda$ is a partition of
$d$ with at most $2$ parts, such that $\lambda$ does not have $p$ equal parts. 
Recall that we   identify $\lambda$ with $m=\lambda_1-\lambda_2$. Hence
unless $p=2$ and $d$ is even, all indecomposable summands of $T$ occur in $V^{\otimes d}$, and then $V^{\otimes d}\domdim_S S=\infty$, by the following:

\begin{Lemma} \label{lem:5.0.1} If $V^{\otimes d}$ has all $T(\lambda)$ as direct summands, then 
	$$\inf\{ V^{\otimes d}\domdim_S M\colon M\in \mathcal{F}(\St) \}=+\infty.$$
\end{Lemma}
\begin{proof}
 Every module in $\mathcal{F}(\St)$ admits a finite $\add T$-coresolution (see for example \citep[Lemma 6]{Ri} or \citep{zbMATH03981431}) which, in particular, remains exact under $\Hom_S(-,V^{\otimes d})$. By Corollary \ref{ddusingcoresolutions}, the result follows.
\end{proof}

\subsection{The characteristic two case}

 Lemma \ref{lem:5.0.1} leaves us to consider $p=2$ and $d$ even $(\neq 0$).  In this case, as mentioned above, the components of $V^{\otimes d}$ are the $T(m)$ with $m\neq 0$. 
The standard sequence (\ref{(4.1)})  is an $\add T$-approximation, this follows from 
a special case of Proposition \ref{prop:2.3.4}. 
In particular,
\begin{align}
	V^{\otimes d}\domdim_S \Delta(d) = 1 + V^{\otimes d}\domdim_S X(d).  \label{eq12}
\end{align}

\begin{Theorem} \label{thm:5.0.2} Let $d=2s> 0$. We have
	$V^{\otimes d}\domdim_S (\Delta({d})) = s + V^{\otimes d}\domdim_S (T(0))$.
	\end{Theorem}

 To prove this, we will use Lemma \ref{lem:2.3.2} on extensions of $\St(t)$ by $\St(2+t)$.

The cases  $d=2$ and $d=4$ are easy.
\begin{enumerate}
	\item For $d=2$ we have the exact sequence $0\to \Delta(2) \to T(2) \to \Delta(0) = T(0)\to 0$, which proves the statement of the Theorem  by Corollary \ref{ddusingcoresolutions}. 
	\item Let $d=4$, we have the exact sequence $0\to \Delta(4) \to T(4) \to \Delta(2)\to 0$. 
	Splicing this with the sequence for $\Delta(2)$ gives the claim.
\end{enumerate}

Degrees $d\geq 6$ need more work.
The main ingredient is the observation that subquotients of the form
$T(2)\otimes N^F$ with $N\in \cF(\Delta)$ are not relevant for
a minimal $V^{\otimes d}$-approximation.

\begin{Lemma}\label{lem:5.0.3} Let $X\in S(2, s)\m$.
Assume  that $X\in \cF(\Delta)$. Then  the module
	$T(2)\otimes X^F$ has infinite relative dominant dimension with respect to $V^{\otimes d}$ for any even degree $d$ greater or equal to $2s+2$.
	\end{Lemma}
\begin{proof}
	By Lemma \ref{lem:2.3.2} it suffices to prove this when $X=\Delta(s)$. 
	We proceed by induction on $s$. When $s=0$ or $s=1$ we see that $T(2)\otimes \Delta(s)^F$ is a summand of $V^{\otimes d}$. 
	For the inductive step consider 
	the exact sequence 
	$$0\to T(2)\otimes \Delta(s)^F  \to T(2)\otimes T(s)^F  \to T(2)\otimes X(s)^F \to 0.$$  
	The middle term is isomorphic to 
	$T(2 + 2s)$. Since $X(s)$ has a filtration with quotients $\Delta(t)$ for $t< s$ it follows by induction
	(and Lemma \ref{lem:2.3.2}) that  the $T(2)\otimes X(s)^F$ has infinite \mbox{$V^{\otimes d}$-dominant dimension} for any even degree $d\geq 2s+2$.
	We deduce that $T(2)\otimes \Delta(s)^F$ has infinite $V^{\otimes d}$-dominant dimension as well.
\end{proof}

\begin{proof}[\textbf{Proof of  Theorem \ref{thm:5.0.2}}]
Assume $d=2s\geq 4$.
 By Proposition \ref{prop:4.3.2}, there exists $S(2, d)$-exact sequences \begin{align}
		0\rightarrow \St(2t+2)\rightarrow T(2)\otimes \St(t)^F\rightarrow \St(2t)\rightarrow 0,
	\end{align} for every $0\leq t\leq s-1$. Moreover, they remain exact over $\Hom_{S(2, d)}(-, V^{\otimes d})$ since $V^{\otimes d}$ is a partial tilting module and so $\Ext_{S(2,d)}^1(\St(2t+2), V^{\otimes d})=0$. By Lemma \ref{lem:5.0.3}, $T(2)\otimes \St(t)^F$ has infinite relative dominant dimension with respect to $V^{\otimes d}$ for $0\leq t\leq s-1$. By Lemma \ref{lem:2.3.2}, \begin{align}
	V^{\otimes d}\domdim_{S(2, d)} \St(2t+2)=1+V^{\otimes d}\domdim_{S(2, d)} \St(2t), \quad 0\leq t\leq s-1.
\end{align}Hence, $V^{\otimes d}\domdim_{S(2, d)} \St(0)=s+V^{\otimes d}\domdim_{S(2, d)} \St(d)$. 
\end{proof}

We will now determine the relative dominant dimension of $S(2,d)$ with respect to $V^{\otimes d}$. 
Let $P_d(m)$ be the indecomposable projective $S(2, d)$-module with homomorphic image $L(m)$. Throughout $d$ and $m$ are even.

\begin{Lemma}\label{lem:5.0.4}
        Consider a projective module $P_d(m)$ where $d$ and $m$ are
        even with $m < d$. Then one of the following holds.
        \begin{enumerate}[(a)]
        	\item The number of  quotients in a $\Delta$-filtration of $P_d(m)$ is even and
        	$P_d(m)$ has a twisted filtration.
        	\item There is an exact sequence
        	$$0\to \Delta(d) \to P_d(m) \to P_{d-2}(m)\to 0$$
        	and $P_{d-2}(m)$ has a twisted filtration.
        \end{enumerate}
\end{Lemma}
\begin{proof} Our strategy consists of proving that the projective module $P_d(m)$ is a quotient of a tilting module and then to combine this fact with Lemma \ref{lem:4.3.3}.
	Let $r_n:= 2^{n+1}-2$.
	There is a unique $n$
	such that $r_{n-1} < d \leq r_{n}$. By our convention,
	we can view $P_d(m)$ as a module in degree $r_n$. Since it has a simple top isomorphic to $L(m)$, it
	is isomorphic to  a quotient of $P_{r_n}(m)$. 
	
	By results in \cite{EH}, \cite{EL}, we have the following.
	\begin{enumerate}[(i)]
		\item If $0\leq m < (r_n)/2$, then $P_{r_n}(m)$ is a tilting module (in
			fact, it is isomorphic to $T(r_n-m)$).
		\item For $(r_n)/2 < m \leq r_n$, the projective module 
		$P_{r_n}(m)$ is a factor module of the tilting module
		$T(r_{n+1}-m)$. 
	\end{enumerate}

	We exploit this now. 
	Let $\hat{m}$ be the weight as above such that $P_d(m)$ is a quotient of $T(\hat{m})$.
	With the notation as in  Lemma \ref{lem:4.3.3},  since $\mathcal{F}(\St)$ is closed under kernels of epimorphisms
	there is a submodule $U\subseteq T(\hat{m})$ which has a $\Delta$-filtration,  with 
	$M_i \subseteq U \subset M_{i-1} $ where $0< i\leq k$, 
	and $P_{d}(m) \cong  T(\hat{m})/U$.
	
	If $U=M_i$ then we have part (a). Otherwise, $P_d(m)$ has 
	the submodule 
	$M_{i-1}/U$ which is isomorphic to $\Delta(2s_i)$ and $U/M_i \cong \Delta(2s_i+2)$.  Moreover, $P_d(m)/\Delta(2s_i)\simeq T(\hat{m})/M_{i-1}$ which has a twisted filtration.
Since $\Delta(2s_i)\subset P_d(m)$ we deduce $2s_i\leq d$. 
	Suppose we have $2s_i < d$, then $2s_i+2 \leq d$.  Hence, $T(\hat{m})/M_i\in S(2, d)\m$. Since $P_d(m)$ is a quotient of the indecomposable $T(\hat{m})$, $T(\hat{m})$ has a simple top isomorphic to $L(m)$.
	So, the module $T(\hat{m})/M_i$ has a simple top isomorphic to $L(m)$ and
	is in degree $d$, and therefore must be a quotient of $P_d(m)$.  In particular, we would obtain $M_i=U$.
	This is not so in the case considered. Therefore $2s_i=d$  and the result follows from (\ref{4.2}).
\end{proof}

\begin{example}\normalfont \label{ex:5.0.5}
Consider Figure 1, with  $d=28$. Then $r_n=30$ and the projective modules
	$P_d(m)$ for $15 < m < 28$ are as follows. 
	We have (a) when $m=18, 20, 22, 26$
	and we have (b) when $m=16, 24$.
	Note that cases $P_d(m) \cong P_{d-2}(m)$ occur in (a). 
\end{example}

\begin{Cor}\label{cor:5.0.6}
	With the setting as in  Lemma \ref{lem:5.0.4}, \\
        if (a) occurs, then
        $V^{\otimes d}\domdim_S P_d(m) = +\infty$. \\
        If (b) occurs, then for $d=2s$ we have
        $V^{\otimes d}\domdim_S P_d(m)  = V^{\otimes d}\domdim_S \Delta(d).$
	In particular, 
$$V^{\otimes d}\domdim_S S(2, d) = V^{\otimes d}\domdim_S \Delta(d) = 
	(d/2)  + V^{\otimes d}\domdim_ST(0).$$
\end{Cor}
\begin{proof}
	This follows directly from Lemma \ref{lem:5.0.4} and Lemma \ref{lem:2.3.2}.
\end{proof}

 This completes the proof of Theorem \ref{thm:2} for algebraically closed fields. By \citep[Lemma 3.2.3]{Cr2}, the result also holds over arbitrary fields.

\subsection{The quantum case}

\begin{Remark}
	If $q$ is not a root of unity then  $S_{K, q}(2, d)$ is semi-simple (\citep[4.3(7)]{Do2}) and $V^{\otimes d}$ being faithful is a characteristic tilting module. Otherwise,  the summands of $V^{\otimes d}$ over $S_{K, q}(2, d)$ are the tilting modules labelled by the $\ell$-regular partitions of $d$ in at most $2$ parts, where $q$ is an $\ell$-root of unity.  Hence, replacing $\characteristic K$ by $\ell$ in Lemma \ref{lem:4.0.1}, we obtain that
	$V^{\otimes d}$ is a characteristic tilting module over $S_{K, q}(2, d)$ if $q+1\neq 0$ or $d$ is odd.
\end{Remark}

For the quantum case, it is enough to take $S=S_{K, q}(2, d)$ where $q+1=0$. In this case, everything is
	exactly the same as over $S(2, d)$ when $\characteristic K=2$. Namely, we may take $S_{K, q}(2, d)$ as $A_q(2, d)^*$ as it is done in \cite{Do2} and \cite{Cox}, and also in \cite{DDo} and
	\cite{Do3}. 	The definition  of  $A_q(2)$ may be found in \citep[p. 16]{DDo}. This means that one takes the quantum group $G(2)$ as defined in \cite{DDo} instead of $SL(2, K)$. 
	
	As it is explained in \citep[Sections 3.1 and 3.2]{EL}, we can use the same labelling for weights, in \cite{EL}; in that paper the parameter $q$ is a primitive $\ell$-th root of $1$ and we only need
	$\ell=2$. 
	We can regard $S_{K, q}(2, d)$ as a factor algebra of $S_{K, q}(2, d+2)$, using \citep[Section 4.2]{Do2}, and therefore regard modules in degree $d$
	again as modules in degree $d'$ for $d'>d$ of the same parity.
	
	There is a Frobenius morphism from the quantum group $G(2)$ to the classical setting, hence if $\Delta(m)$ (resp. $T(m)$)  is a  standard module (resp. a tilting module)
	for the classical setting, then $\Delta(m)^F$ (resp. $T(m)^F$) is a module for the quantum group, and so are the tensor products
	$T(2)\otimes \Delta(m)^F$ and $T(2)\otimes T(m)^F$ modules for the quantum group.  
	The $q$-analogues of the exact sequences in Subsection \ref{Twisted tensor product methods} and Proposition \ref{prop:4.3.2}
	exist  by \citep[Prop. 3.3 and 3.4.]{Cox}. See also \citep[Proposition 3.1]{EL} (our situation of interest is recovered by fixing $l=2$ in their setup). The $q$-analogue of Proposition \ref{prop:4.3.1} can be found in \citep[Section 3.4, page 73, (8)]{Do2}. 
	
	We note that \ref{4dot3dot1} of Subsection \ref{Twisted tensor product methods} and the $q$-analogue imply, by induction, that all decomposition numbers are $0$ or $1$.

	In \cite{DDo} it is shown that this version of 
	the $q$-Schur algebra is the same as our definition, as the endomorphism algebra
	of the action of the Iwahori-Hecke algebra on the tensor space $V^{\otimes d}$, see
	Section \ref{sec6} below. The definition of the Iwahori-Hecke algebra, as we take it is given in \ref{Theqanalogue} below. In particular, we denote by $H$ the Iwahori-Hecke algebra. Their strategy in \citep[Section
	3]{DDo}  is to show that
	the action of the Iwahori-Hecke algebra on the tensor space
	is a comodule homomorphism (see 3.1.6 of \cite{DDo}). The $H$-action
	in \citep[3.1.6]{DDo}
	is not the same as ours, but
	it is explained in detail (see 4.4.3 of \cite{DDo}) that the action we
	use also can be taken.

Hence, the arguments of Section \ref{dominant dimension of the regular module} remain valid in the quantum case and therefore, we obtain the following:

 \begin{Theorem}\label{thm:5.0.7} Let $K$ be a field and fix $q=u^{-2}$ for some $u\in K$. Let $S$ be the $q$-Schur algebra $S_{K, q}(2, d)$ and $T$ be the characteristic tilting module of $S$. Then,
	\begin{align*}
		V^{\otimes d}\domdim_S S=2\cdot V^{\otimes d}\domdim_S T=\begin{cases}
			d, & \text{ if } 1+q=0 \text{ and } d \text{ is even}, \\
			+\infty, &\text{ otherwise }
		\end{cases}.
	\end{align*}
\end{Theorem}

\begin{Remark}
	One might want to know for which $m$ it is true that $V^{\otimes d}\domdim_S P_d(m)$ is finite.
 In principle, one can answer this, using the formula in \cite{H} for decomposition numbers. Namely 
	this $V^{\otimes d}-$ dominant dimension  is finite if and only the number of $\Delta$-quotients of $P_d(m)$ is odd, ie the number of $1$s in the
column of $L(m)$.
\end{Remark}

\section{Temperley-Lieb algebras}\label{sec6}

These algebras were introduced as a model for statistical mechanics (\citep{zbMATH03335816}), and then became
popular through the work of Jones. In particular, he discovered that they occur as quotients of
Iwahori-Hecke algebras (\citep{zbMATH03802219, zbMATH03899758}). See also \cite{zbMATH00795321} for further details.
We give the definition and discuss the connections with Schur algebras.

\begin{Def}  Let $R$ be a commutative ring and $\delta$ an element of $R$. 
	The \emph{Temperley-Lieb algebra}
	$TL_{R,d}(\delta)$ over $R$ is the $R$-algebra generated by elements $U_1, U_2, \ldots, U_{d-1}$ with defining relations, here $1\leq i, j\leq d-1$
	such that each term is defined:
	\begin{enumerate}[(a)]
		\item $U_iU_j=U_jU_i$ ( $|i-j|>1$),
		\item  $U_i^2 = \delta U_i$,
			\item $U_{i}U_{i+1}U_i = U_i$, $1\leq i\leq n-2,$
			\item  $U_{i}U_{i- 1}U_i = U_i$, $2\leq i\leq n-1.$
	\end{enumerate}
\end{Def}

It can be viewed as a diagram algebra, with a very 
extensive  literature, but we will not give details since we do 
use diagram calculations.

\subsection{The classical case} We will start by considering the class of Temperley-Lieb algebras which can be viewed as quotients of group algebras of the symmetric group.

\begin{Lemma}\label{lem:6.1.1}
	There is a surjective algebra homomorphism $\Phi: R\cS_d \to TL_{R,d}(-2)$ taking the generator $T_i = (i \ i+1)$ of $\cS_d$ 
	to $U_i+1$ for $1\leq i\leq d-1$.
\end{Lemma}
\begin{proof}
	Recall that the group algebra $R\cS_d$ is generated by the $T_i$ subject to the relations
	\begin{enumerate}[(a)]
		\item $T_iT_{i+1}T_i = T_{i+1}T_iT_{i+1}$,
		\item $T_iT_j = T_jT_i$ \ ($|i-j| > 1$),
		\item $T_i^2=1$,
	\end{enumerate}
	for $1\leq i, j \leq d-1$ such that each factor is defined.
	To show that the map is well-defined one has to check that it preserves these relations; this is straightforward.
	It is clear that $\Phi$ is surjective, noting that $U_i = \Phi(T_i-T_i^2)$.
\end{proof}

The following description of the kernel of $\Phi$ goes back to [Jon87, p. 364].

\begin{Theorem}\label{thm:6.1.2} For  each $i=1, 2, \ldots, d-2$ define
	$$x_i:= T_iT_{i+1}T_i - T_iT_{i+1} - T_{i+1}T_i + T_i + T_{i+1} -1 \in R\cS_d.
	$$ Let $I$ be the ideal of
	$R\cS_d$ generated by the $x_i$ for $1\leq i\leq d-2$. Then there
	is an exact sequence
	$$0\to I \to R\cS_d \stackrel{\Phi}\to TL_{R,d}(-2) \to 0
	$$
\end{Theorem}

\begin{proof} One checks that $\Phi(x_i)=0$ for each $i=1, \ldots, d-2$. So we have a commutative diagram
	$$\CD 0@>>> {\rm ker} \Phi @>>> R\cS_d  @>{\Phi}>> TL_{R, d}(-2) @>>> 0\cr
	&&@A{\iota}AA    @A{id_{R\cS_d}}AA @A{\pi}AA \cr
	0  @>>> I @>>> R\cS_d @>>> R\cS_d/I  @>>> 0
	\endCD \ ,$$
	where $\pi$ maps the image of $T_i$ in $RS_d/I$ to $\Phi(T_i) = U_i+1$, and $\iota$ is the inclusion map.
	Consider $\pi': TL_{R, d}(-2) \to R\cS_d/I$ defined by taking $U_i$ to the image of $T_i-1$ in $R\cS_d/I$. One checks that
	$\pi'$ preserves the defining relations for  $TL_{R, d}(-2)$, so that it is a well-defined map.
	Finally,
	$$\pi'(\pi(T_i + I)) = T_i+I, \ \ \pi(\pi'(U_i)) = U_i.
	$$
	Therefore ${\rm ker} \Phi = I$.
\end{proof}

It is nowadays widely known that Temperley-Lieb algebras can be viewed as the centraliser algebras of quantum groups $\mathfrak{sl}_2$ in the endomorphism algebra of a tensor power and it goes back to the work of Martin \cite{zbMATH01235858} and Jimbo \cite{zbMATH03970994}. 
Recall that over $R$, the Schur algebra $S_R(2, d)$ is defined as the endomorphism algebra $\End_{R\cS_d}((R^2)^{\otimes d})$, where $(R^2)^{\otimes d}$ affords a right $R\cS_d$-module structure via place permutation.
In order to relate the
Temperley-Lieb algebra to the Schur algebra, we need a suitable action of the Temperley-Lieb algebra on the tensor space.
It is as follows.

\begin{Theorem}\label{thm:6.1.3} Let $V$ be a free $R$-module of rank $2$.
	Then $V^{\otimes d}$ is a module over $\La= TL_{R, d}(-2)$ where $U_i$ acts as ${\rm id}_V^{\otimes(i-1)}\otimes \tau \otimes {\rm id}_V^{\otimes (d-i-1)}$. Here
	$\tau$ is the endomorphism of $V^{\otimes 2}$ defined by 
	$$\tau(v_1\otimes v_2) = v_2\otimes v_1 - v_1\otimes v_2$$
	(for $v_1, v_2\in V$). Moreover there is an algebra isomorphism
	$$\La \to {\rm End}_{S_R(2, d)}(V^{\otimes d})^{op}.$$
\end{Theorem}

\begin{proof}   We know that $R\cS_d$ acts by place permutations on $V^{\otimes d}$ and we can view this  as a right action. 
	With this, $T_i- 1$ acts exactly as the action of $U_i$ as in the statement. This shows that it factors through $\La$. In particular, to show that $\Lambda\rightarrow {\rm End}_{S(2, d)}(V^{\otimes d})^{op}$ is surjective it is enough to check that the canonical map $RS_d\rightarrow {\rm End}_{S(2, d)}(V^{\otimes d})^{op}$ is surjective. But this follows from classical Schur--Weyl duality (see for example \cite{KSX}). In fact, this can be seen in the following way: let $K$ be a field, then the canonical map $K\cS_d\rightarrow \End_{S(2, d)}(V^{\otimes d})^{op}$ fits in the following commutative diagram
	\begin{equation*}
		\begin{tikzcd}
			K\cS_d \arrow[rr] \arrow[dr, "\psi"]&  &\End_{S(2, d)}(V^{\otimes d})^{op}
			\\
			& \End_{S(d, d)}((K^d)^{\otimes d})^{op} \arrow[ur, "\phi"] & 
		\end{tikzcd}
	\end{equation*}
	Here, $\psi$ is surjective because $\dd S_K(d, d)\geq 2$ and $\phi$ is surjective by \citep[1.7]{E1} and \citep[4.7]{Do2} because $(K^d)^{\otimes d}$ is a projective-injective module over $S(d, d)$.  Observe that $\End_{S_R(2, d)}(V^{\otimes d})^{op}\in R\proj$ (see for example \citep[Proposition A.4.3., Corollary A.4.4.]{Cr1}) and it has a base change property (see for example \citep[Corollary A.4.6]{Cr1}). In particular, $R(\mi)\otimes_R \End_{S_R(2, d)}(V^{\otimes d})^{op}\simeq \End_{S_{R(\mi)}(2, d)}(V^{\otimes d})^{op}$ for every maximal ideal $\mi$ of $R$. By the above discussion, the maps $R(\mi)S_d\rightarrow \End_{S_{R(\mi)}(2, d)}(V^{\otimes d})^{op}$ are surjective for every maximal ideal $\mi$ of $R$.  Now, by Nakayama's Lemma the map $RS_d\rightarrow \End_{S_R(2, d)}(V^{\otimes d})^{op}$ is surjective.

	It remains to show that the action of $\La$ is injective. 
	Let $\sum_i a_iU_i \in \La$ acting as zero on $V^{\otimes d}$. The action of $\sum_i a_iU_i $ in $y_k$, defined as the basis element $e_1\otimes \ldots \otimes e_1\otimes e_2\otimes e_1\otimes \ldots \otimes e_1$ where $e_2$ appears in position $k+1$, yields that $a_k=a_{k+1}=0$. This concludes the proof.
\end{proof}

\subsection{The $q$-analogue}\label{Theqanalogue}We shall now discuss the general case of Theorem \ref{thm:6.1.3} and  its importance for all Temperley-Lieb algebras.

Let $R$ be a commutative Noetherian ring with an invertible element $u\in R$. We fix a natural number $d$, and we set
$q:= u^{-2}$.
We take the \emph{Iwahori-Hecke algebra} $H=H_{R, q}(d)$ to be the $R$-algebra with basis $\{  \wT_{w}\mid  w\in \cS_d\}$ with relations
$$\wT_w\wT_s = \left\{\begin{array}{ll} \wT_{ws} & \mbox{if } l(ws) = l(w) + 1\cr
	(u-u^{-1})\wT_w + \wT_{ws} & \mbox{ otherwise.}
\end{array}
\right.
$$
here $s$ runs through the set of transpositions $S=\{ (i \ i+1)\mid 1\leq i < d\}$ in $\cS_d$, and where $l(w)$ is the usual length for $w\in \cS_d$, that
is the minimal number of transpositions needed in a factorisation of $w$.

This presentation corresponds to the presentation used in \cite{DJ1, DJ2}, \cite{DDo}, \cite{Do2} by
$$\wT_w = (-u)^{l(w)}T_w.$$
The algebra $H$ can also be defined by
the braid relations, together with
$\wT_s^2 = (u-u^{-1})\wT_s + 1$, that is
$$(\wT_s-u)(\wT_s+u^{-1}) = 0 \ \ (s\in S).
$$

\begin{Lemma} \label{lem:6.2.1} Let $\delta = -u-u^{-1}$. Then 
	there is a surjective algebra homomorphism 
	$$\Phi: H_{R, q}(d)  \to TL_{R,d}(\delta)$$ 
	taking the generator $\wT_i:=\wT_{(i \ i+1)}$ to $U_i+ u$ for $1\leq i \leq d-1$.
\end{Lemma}

\begin{proof} We must show that $\Phi$ is well-defined, that is
	it respects the relations of the Hecke algebra.
	It is clearly surjective.
	We work with the presentation via 
	the braid relations, together with
	$(\wT_i-u)(\wT_i+u^{-1}) = 0$ (for $1\leq i < d-1$).
	With the definition given, $\Phi(\wT_i + u^{-1}) = U_i + u+ u^{-1} = U_i-\delta$ and $\Phi(\wT_i-u) = U_i$. 
	Hence we have  $\Phi(\wT_i + u^{-1})\Phi(\wT_i - u) = (U_i-\delta)U_i = 0$.
	
	To check the braid relations, we compute
	$$(U_i+u)(U_{i+1}+u)(U_i+u) = U_i + (u\delta)U_i + u(U_iU_{i+1} + U_{i+1}U_i) 
	+ 2u^2U_{i} + u^2U_{i+1} + u^3
	$$
	The coefficient of $U_i$ is equal to $u^2$. With this, the expression is
	symmetric in $i, i+1$ and is therefore equal to
	$(U_{i+1}+u)(U_i+u)(U_{i+1}+u)$.
\end{proof}

We want to  determine the kernel of $\Phi$.

\begin{Theorem} \label{thm:6.2.2} For  each $i=1, 2, \ldots, d-2$ define
	$$x_i:= \wT_i\wT_{i+1}\wT_i -u\wT_i\wT_{i+1} - u\wT_{i+1}\wT_i + u^2\wT_i + u^2\wT_{i+1} -u^3 \in H.
	$$ Let $I$ be the ideal of
	$H_{R, q}(d)$ generated by the $x_i$ for $1\leq i\leq d-2$. Fix $\delta = -u-u^{-1}$, then there
	is an exact sequence
	$$0\to I \to H_{R, q}(d)  \stackrel{\Phi}\to TL_{R,d}(\delta) \to 0.
	$$
\end{Theorem}

\begin{proof}  From the  proof of Lemma \ref{lem:6.2.1} we see that
	$$\Phi(\wT_i)\Phi(\wT_{i+1})\Phi(\wT_i) =  u^2(U_i+U_{i+1}) + u(U_{i+1}U_i + U_iU_{i+1}) + u^3$$
	and with this one gets that $\Phi(x_i) = 0$ for all $i$. Hence
	$I\subseteq {\rm ker}(\Phi)$. 
	Analogously to the proof of Theorem 
	\ref{thm:6.1.2} replacing the map $\pi'$ with the map $TL_{R, d}(\delta)\rightarrow H_{R, q}(d)$ defined by taking $U_i$ to the image of $\wT_i-u$ in $H_{R, q}(d)/I$ one proves equality.
\end{proof}

Let $V^{\otimes d}$ be the free $R$-module of rank $n$ over $R$ (later we will take $n=2$). Then $V^{\otimes d}$ is a right
$H$-module, which can be thought of as a deformation of the place permutation action of $\cS_d$. Denote by $I(n, d)$ the set of maps $\{1, \ldots, d\}\rightarrow \{1, \ldots, n\}$ and by $i_j$ the image $i(j)$. If ${\bf i} \in I(n,d)$  labels the basis element
$e_{\bf{i}} = e_{i_1}\otimes e_{i_2}\otimes \ldots \otimes e_{i_d}$ of $V^{\otimes d}$ and $s=(t \ t+1)\in S$ we write $e_{\bf i}\cdot s$ for the basis element obtained by interchanging $e_{i_t}$ and $e_{i_{t+1}}$. Then
$$e_{\bf{i}} \cdot \wT_s:= \left\{\begin{array}{ll} 
	e_{\bf{i}}\cdot s & \ \  i_t < i_{t+1} \cr
	ue_{\bf{i}} & \ \ i_t = i_{t+1} \cr
	(u-u^{-1})e_{\bf{i}} + e_{\bf{i}}\cdot s &  \ \ i_t > i_{t+1}
\end{array}
\right.
$$
Focussing on the TL algebra, we take $n=2$.
Recall that the $q$-Schur algebra $S_q(2, d)$ is the endomorphism algebra
${\rm End}_H(V^{\otimes d})$ via the action as above.

\begin{Theorem}\label{thm:6.2.3} The $H$-module structure on  $V^{\otimes d}$ factors
	through 
	$\Phi: H \to \La= TL_{R, d}(\delta),$ where $\delta = -u-u^{-1}$. Hence $U_s$ acts as
	$$e_{\bf i} U_s:= \left\{\begin{array}{ll} e_{{\bf i}s}- ue_{\bf i} & i_t< i_{t+1} \cr
		ue_{\bf i} - ue_{\bf i} & i_t=i_{t+1}\cr
		(-u)^{-1}e_{\bf i} + e_{{\bf i}s} & i_t> i_{t+1}
	\end{array}
	\right.$$
	where $s=(t \ t+1)\in S$. 
	Moreover there is an algebra isomorphism
	$$\La \to {\rm End}_{S_{R, q}(2, d)}(V^{\otimes d})^{op}.$$
\end{Theorem}
\begin{proof}The first statement follows by checking that the elements $x_i$ act as zero on $V^{\otimes d}$.
	
	The element $\wT_i-u$ acts exactly as the action of $U_i$ in $V^{\otimes d}$, so the canonical map \linebreak${H_{R, q}(d)\rightarrow \End_{S_{R, q}(2, d)}(V^{\otimes d})^{op}}$ factors through $\Lambda$, that is, there is an algebra homomorphism $\La \to {\rm End}_{S_{R, q}(2, d)}(V^{\otimes d})^{op}.$ The same argument as the one given in Theorem \ref{thm:6.1.3} works in this case replacing the Schur algebra by the $q$-Schur algebra and the group algebra of the symmetric group by the Iwahori-Hecke algebra. The injectivity follows again by considering the action of the elements in $\Lambda$, acting as zero on $V^{\otimes d}$, on the elements $y_k$ defined in the exactly same way as in Theorem \ref{thm:6.1.3}.
\end{proof}
Theorem \ref{thm:6.2.3} places $V^{\otimes d}$ in a central position in the representation theory of Temperley-Lieb algebras where it plays a role similar to that played by $(R^n)^{\otimes d}$ in the study the representation theory of symmetric groups via Schur algebras. In fact, Theorem 8.1.5 of \citep{Cr2} specializes to the following.

	\begin{Cor} \label{cor:6.2.4} Let $K$ be a field and fix $q=u^{-2}$ for some element $u\in K^\times$. Let $T$ be a characteristic tilting module of 
		$S$ and let $R(S)$ be the Ringel 
		dual of $S:=S_{K, q}(2, d)$ over a field $K$. 
		Then, $(R(S), \Hom_S(T, V^{\otimes d})$ is a $(V^{\otimes d}\domdim_S T-2)$-$\mathcal{F}(\St_{R(S)})$ quasi-hereditary cover of 
		$TL_{K, d}(-u-u^{-1})$, where $\St_{R(S)}$ denotes the set of standard modules over $R(S)$.
		Moreover, the following assertions hold:
		\begin{enumerate}[(i)]
			\item If $q+1\neq 0$ or $d$ is odd, then $TL_{K, d}(-u-u^{-1})$ is the Ringel dual of $S_{K, q}(2, d)$, and in particular, it is a split quasi-hereditary algebra over $K$;
			\item If $q+1=0$ and $d$ is even, then $(R(S), \Hom_S(T, V^{\otimes d})$ is a $(\frac{d}{2}-2)$-$\mathcal{F}(\St_{R(S)})$ quasi-hereditary cover of 
			$TL_{K, d}(0)$ and $\HN_F \mathcal{F}(\St_{R(S)})=\frac{d}{2}-2$. In particular,
the Schur functor $F:=\Hom_{R(S)}(\Hom_S(T, V^{\otimes d}), - )\colon R(S_{K, q}(2, d))\m\rightarrow TL_{K, d}(0)\m$ induces bijections
			\begin{equation*}
				\Ext_{R(S)}^i(M, N)\simeq \Ext_{TL_{K, d}(0)}^i(FM, FN), \quad \forall M, N\in \mathcal{F}(\St_{R(S)}), \quad  0\leq i\leq \frac{d}{2}-2.
			\end{equation*}
		\end{enumerate}
	\end{Cor}
\begin{proof}
	The result follows from Theorem \ref{thm:5.0.7} and \citep[Theorem 8.1.5.]{Cr2} and \citep[Theorem 6.0.1]{Cr2}.
\end{proof}

\section{Uniqueness of the quasi-hereditary cover of $TL_{R, d}(\delta)$}\label{sec7}

In Corollary \ref{cor:6.2.4}, we construct a quasi-hereditary cover of $TL_{K, q}(\delta)$ using the Ringel dual of a $q$-Schur algebra. We will argue now that it is the best quasi-hereditary cover of $TL_{K, d}(\delta)$ if $d>2$. For that, going to the integral case is helpful. Assume that $R$ is a commutative Noetherian ring. Let $u$ be an invertible element of $R$ and fix $q=u^{-2}$. If $d=1, 2$ then the Temperley-Lieb algebra $TL_{R, q}(-u-u^{-1})$ coincides with the Iwahori-Hecke algebra $H_{R, q}(d)$ and so this case was dealt in \citep[Subsection 7.2]{Cr1}.

Assume from now on that $d>2$. Combining Theorem \ref{thm:6.2.3} with Theorem 8.1.5 of \citep{Cr2} we obtain the following:

\begin{Cor} \label{cor:7.1}
	Let $R$ be a commutative Noetherian ring. Fix an element $u\in R^\times$ and $q=u^{-2}$. 
	Let $T$ be a characteristic tilting module of $(S_{R, q}(2, d), \{\St(\l)_{\l\in \L^+(2, d)} \})$. Denote by $R(S)$ the Ringel dual of $(S_{R, q}(2, d), \{\St(\l)_{\l\in \L^+(2, d)} \})$, that is, $R(S)=\End_{S_{R, q}(2, d)}(T)^{op}$. 
	
	Then, $(R(S), \Hom_{S_{R, q}(2, d)}(T, V^{\otimes d}) )$ is a $(V^{\otimes d}\domdim_{S_{R, q}(2, d), R} T-2)-\mathcal{F}(\Stsim_{R(S)})$ split quasi-hereditary cover of $TL_{R, d}(-u-u^{-1})$.
\end{Cor}

In the following, we will write $R(S)$ to denote the Ringel dual $\End_{S_{R, q}(2, d)}(T)^{op}$. 
Denote by $F_{R, q}$ the Schur functor associated with the quasi-hereditary cover constructed in Corollary \ref{cor:7.1}. The aim now is to compute $\HN_{F_{R, q}}\mathcal{F}(\Stsim_{R(S)})$ and in particular to determine $V^{\otimes d}\domdim_{S_{R, q}(2, d), R} T$ in terms of the ground ring $R$.

\begin{Theorem}\label{thm:7.2}
	Let $R$ be a commutative Noetherian ring. Fix an element $u\in R^\times$ and $q=u^{-2}$.  Let $T$ be a characteristic tilting module of $(S_{R, q}(2, d), \{\St(\l)_{\l\in \L^+(2, d)} \})$. Then,
	\begin{align*}
		V^{\otimes d}\domdim_{(S_{R, q}(2, d), R)} T=\begin{cases}
			\dfrac{d}{2}, & \text{ if } 1+q\notin R^\times \text{ and } d \text{ is even}\\
			+\infty, & \text{ otherwise}
		\end{cases}.
	\end{align*}
\end{Theorem}
\begin{proof}Since $S_{R, q}(2, d)$ has the base change property: $S\otimes_R S_{R, q}(2, d)\simeq S_{S, 1_S\otimes q}(2, d)$ as $S$-algebras for every commutative ring $S$ which is an $R$-algebra and the standard modules of $S_{S, 1_S\otimes q}(2, d)$ are of the form $S\otimes_R \St(\l)$, $\l\in \L^+(n, d)$ (see for example \citep[Subsection 3.3, Section 5]{cruz2021cellular}), the result follows from Theorem \ref{thm:5.0.7}, \citep[Propositions A.4.7, A.4.3.]{Cr1} and \citep[Theorem 3.2.5.]{Cr2}.
\end{proof}

\subsection{Hemmer-Nakano dimension of $\mathcal{F}(\Stsim_{R(S)})$} Similarly to the classical case (see also \cite{Cr1}), there are two cases to be considered.  \label{sec7.1}

Following \cite{Cr1}, the commutative Noetherian ring $R$ is called \emph{$2$-partially $q$-divisible} if $1+q\in R^\times$ or $1+q=0$.
\begin{Theorem}
		Let $R$ be a local regular $2$-partially $q$-divisible (commutative Noetherian) ring, where $q=u^{-2}$, $u\in R^\times$.  Let $T$ is a characteristic tilting module of $S_{R, q}(n, d)$. Then,
		\begin{align}
		\HN_{F_{R, q}} \mathcal{F}(\Stsim_{R(S)})=V^{\otimes d}\domdim_{(S_{R, q}(2, d), R)} T -2.
		\end{align}
\end{Theorem}
\begin{proof}
	By Corollary \ref{cor:7.1}, $\HN_{F_{R, q}} \mathcal{F}(\Stsim_{R(S)})\geq V^{\otimes d}\domdim_{(S_{R, q}(2, d), R)} T -2$.
	
	 If ${V^{\otimes d}\domdim_{(S_{R, q}(2, d), R)} T}=+\infty$, then $d$ is odd, and then there is nothing to prove. Assume that it is finite. By Theorem \ref{thm:7.2}, $V^{\otimes d}\domdim_{(S_{R, q}(2, d), R)} T=\frac{d}{2}$. In particular, $d$ is even and $1+q\notin R^\times$. Hence, $1+q$ must be zero. Therefore, $${Q(R)\otimes_R V^{\otimes d}\domdim_{(S_{Q(R), q}(2, d), Q(R))} Q(R)\otimes_R T}=\frac{d}{2},$$  where $Q(R)$ is a quotient field of $R$. 
	
	By  \citep[Corollary 5.3.6.]{Cr2}, $\HN_{Q(R)\otimes_R F_{R, q}} \mathcal{F}(Q(R)\otimes_R \St_{R(S)})$ cannot be higher than \linebreak ${V^{\otimes d}\domdim_{(S_{R, q}(2, d), R)} T-2}$. It follows that 
	\begin{align*}
	V^{\otimes d}\domdim_{(S_{R, q}(2, d), R)} T-2=	\HN_{Q(R)\otimes_R F_{R, q}} \mathcal{F}(Q(R)\otimes_R \St_{R(S)})\geq \HN_{F_{R, q}} \mathcal{F}(\Stsim_{R(S)}).
	\end{align*}
\end{proof}

\begin{Theorem}\label{thm:7.1.2}
	Let $R$ be a local regular commutative Noetherian ring which is not a $2$-partially $q$-divisible commutative ring, where $q=u^{-2}$, $u\in R^\times$.  Let $T$ is a characteristic tilting module of $S_{R, q}(n, d)$. Then,
	\begin{align}
		\HN_{F_{R, q}} \mathcal{F}(\Stsim_{R(S)})=V^{\otimes d}\domdim_{(S_{R, q}(2, d), R)} T -1.
	\end{align}
\end{Theorem}
\begin{proof}
	By Corollary \ref{cor:7.1}, if $V^{\otimes d}\domdim_{(S_{R, q}(2, d), R)} T$ is infinite, then there is nothing to show. So, assume that $V^{\otimes d}\domdim_{(S_{R, q}(2, d), R)} T$ is finite. By Theorem \ref{thm:7.2}, $d$ is even and $1+q\notin R^\times$. By assumption, $1+q\neq 0$, otherwise $R$ would be a $2$-partially $q$-divisible ring. It follows that $Q(R)\otimes_R V^{\otimes d}\domdim_{(Q(R)_\otimes R S_{R, q}(2, d), R)} Q(R)\otimes_R T$ is infinite by Theorem \ref{thm:7.2}. The result for $d=2$ follows from \citep[Theorem 7.2.7]{Cr1}. Assume that $d\geq 4$. By Corollary \ref{cor:7.1} and \citep[Theorem 3.2.5.]{Cr2}, $\HN_{F_{R(\mi), q_\mi}} \mathcal{F}(R(\mi)\otimes_R \St_{R(S)})\geq 0$ and ${\HN_{Q(R)\otimes_R F_{R, q}} \mathcal{F}(Q(R)\otimes_R \St_{R(S)})}=+\infty$, where $\mi$ is the unique maximal ideal of $R$, and $q_\mi$ is the image of $q$ in $R/\mi$.
By \citep[Theorem 5.0.9]{Cr1}, $\HN_{F_{R, q}} \mathcal{F}(\Stsim_{R(S)})\geq \frac{d}{2}-1$. The Hemmer-Nakano dimension cannot be higher because similarly to the proof of Theorem 7.2.7 of \cite{Cr1} there exists a prime ideal of height one $\pri$ such that $1+q\in \pri$. Hence, $Q(R/\pri)\otimes_R V^{\otimes d}\domdim_{S_{Q(R/\pri), q_\pri}(2, d), Q(R/\pri)}$ is exactly $\frac{d}{2}$, where $q_\pri$ denotes the image of $q$ in $R/\pri\subset Q(R/\pri)$. The result follows from  \citep[Theorem 5.1.1]{Cr1}.
\end{proof}

\subsection{Uniqueness}\label{Uniqueness}
In this part, assume that $R=\mathbb{Z}[x, x^{-1}]$ and fix $q=x^{-2}$. Assume that $d>2$. By \citep[Proposition 5.0.3]{Cr1} and Theorem \ref{thm:7.1.2}, $\HN_{F_{R, q}} \mathcal{F}(\Stsim_{R(S)})\geq \frac{d}{2}-1$. In particular, the Schur functor $F_{R, q}$ induces an exact equivalence 
\begin{align}
	\mathcal{F}(\Stsim_{R(S)})\rightarrow \mathcal{F}(F_{R, q}\Stsim_{R(S)}). \label{eq15}
\end{align}
\begin{Cor}\label{cor:7.2.1}
	$(R(S), \Hom_{S_{R, q}(2, d)}(T, V^{\otimes d}) )$  is the unique split quasi-hereditary cover of $TL_{R, d}(-x-x^{-1})$ satisfying the property (\ref{eq15}), where $T$ is a characteristic tilting module of $S_{R, q}(2, d)$ and $R(S)$ denotes the Ringel dual of $S_{R, q}(2, d)$. In particular, $TL_{R, d}(-x-x^{-1})$ is a split quasi-hereditary algebra over $R$ if and only if $d$ is odd.
\end{Cor}
\begin{proof}
	The first statement follows from Corollary \ref{cor:2.5.2} together with \citep[Proposition 5.0.3]{Cr1} and Theorem \ref{thm:7.1.2}. For the second statement see for example \citep[Proposition A.4.7.]{Cr1} or \citep[Theorem 6.0.1]{Cr2} together with Theorem \ref{thm:7.2}).
\end{proof}

As a consequence, when $d$ is odd, the Temperley-Lieb algebra $TL_{R, d}(-x-x^{-1})$ is exactly a Ringel dual of $S_{\mathbb{Z}[x, x^{-1}], q}(2, d)$.

\section*{Acknowledgements}
Part of the collaboration was done during a research visit of the second-named author to the University of Stuttgart,  and the authors are grateful for the support given. The authors would like to thank Steffen Koenig for his comments on an earlier version of this manuscript.

\enlargethispage{1\baselineskip}

\interlinepenalty=1000000

\bibliographystyle{alphaurl}
\bibliography{ref}

\end{document}